\documentclass{article}
\usepackage{latexsym}
\usepackage{leftidx}
\usepackage{slashed}
\usepackage{amsmath,amsthm, pdfpages,slashed,color,pdfsync}
\usepackage{enumerate}
\usepackage{amssymb}
\usepackage{upgreek}
\usepackage{tikz}
\RequirePackage{float}
\usepackage{wrapfig}

\usepackage[margin=1in,dvips]{geometry}

\newcommand{\vol}{\textnormal{vol}}

\def\f12{\frac 1 2}

\def\a{\alpha}
\def\b{\beta}
\def\ga{\gamma}
\def\Ga{\Gamma}

\def\ep{\epsilon}
\def\la{\lambda}

\def\Si{\Sigma}
\def\om{\omega}
\def\Om{\Omega}

\def\H{\mathcal{H}}
\def\Hb{\underline{\mathcal{H}}}

\def\Lb{\underline{L}}

\def\pa{\partial}
\def\les{\lesssim}
\def\div{\text{div}}

\def\cD{\mathcal{D}}
\def\cE{\mathcal{E}}
\def\tcD{\widetilde{\cD}}

\def\tv{\widetilde{v}}
\def\tSi{\widetilde{\Si}}

\def\dvol{\textnormal{dvol}}

\newcommand{\nabb}{\mbox{$\nabla \mkern-13mu /$\,}}

\newtheorem{thm}{Theorem}

\newtheorem{Prop}{Proposition}[section]

\newtheorem{Lem}{Lemma}[section]

\newtheorem{Cor}{Corollary}[section]
\newtheorem{Remark}{Remark}[section]

\allowdisplaybreaks

\begin{document}

\title{Asymptotic decay for defocusing semilinear wave equations on Schwarzschild spacetimes}
\author{He Mei, Dongyi Wei and Shiwu Yang}

\date{}

\maketitle

\begin{abstract}
In this paper, we study the long time dynamics of solutions to the defocusing semilinear wave equation $\Box_g\phi=|\phi|^{p-1}\phi$ on the Schwarzschild black hole  spacetimes. For $\frac{1+\sqrt{17}}{2}<p<5$ and sufficiently smooth and localized initial data, we show that the solution decays like $|\phi|\lesssim t^{-1+\epsilon}$ in the domain of outer communication. The proof relies on the $r$-weighted vector field method of Dafermos-Rodnianski together with the Strichartz estimates for linear waves by Marzuola-Metcalfe-Tataru-Tohaneanu.
\end{abstract}

\section{Introduction}

This paper is devoted to the study of the global behaviors of solutions to the energy subcritical  defocusing semilinear wave equation
\begin{equation}
  \label{maineq}
  \Box_g\phi=|\phi|^{p-1}\phi,\quad 1<p<5.
\end{equation}
on the exterior region of the Schwarzschild black hole spacetimes with metric
 \begin{equation*}
  g=-(1-\mu )dt^{2}+(1-\mu)^{-1}dr^2 +r^2d\omega^2,\quad \mu =\frac{2M}{r}
\end{equation*}
in Boyer-Lindquist coordinates $(t,r,\omega)$.

 The Schwarzschild family of black hole spacetimes parameterized by the mass $M>0$ are the first nontrivial solutions to the vacuum Einstein equation.  Motivated by the celebrated nonlinear stability of black hole spacetimes (see recent breakthrough and related discussions in \cite{Dafermos21:Stab:Schwarz}, \cite{Klainerman22:Stab:Kerr}), numerous new ideas and methods are developed to study linear and nonlinear fields on various black hole backgrounds in the past decades. The classical result of  Kay-Wald \cite{Wald87:Schwarz} shows the uniform boundedness of linear waves on Schwarzschild spacetimes.  A significant breakthrough contributed by Dafermos-Rodnianski \cite{dr3} in this direction is the quantitative inverse polynomial decay for linear waves, which plays a crucial role in addressing the full nonlinear problems. Improved decay estimates including the celebrated Price's law \cite{price72:redshift} for linear wave equations could be found for example in
    \cite{Stefanos18:Latetime},
\cite{stefanos18:vector},
  \cite{Soffer11:Price:Schwarz},
\cite{improvLuk},
\cite{Masy21:Price:spins},
  \cite{Tataru12:price}. 
 Different methods and techniques have also been developed to study the nonlinear wave equations on Schwarzschild and Kerr black hole spacetimes
  \cite{Blue06:semiNLW:Schwarz},
  \cite{Dafermos22:NW:Kerr},
  \cite{Mihai18:NW:Schwarz},
  \cite{Luk13:NW:Kerr},
  \cite{Igor05:sNW:BH},
  \cite{Mihai12:Strichartz:Kerr},
 \cite{Mihai22:NW:Kerr}.
 In particular the Strichartz type estimate was used in
\cite{Lai23:Strauss:Schwarz},
\cite{Sogge14:Strauss:Kerr}
 to address the Strauss conjecture for the nonlinear model closely related to \eqref{maineq}. We also mention that there are large global dispersive solutions to a class of nonlinear wave equations on Schwarzschild spacetimes \cite{Jinhua22:NW:Schwarz}.

Most of the aforementioned results   heavily relied on the associated linear theory or perturbative argument. The present paper aims to demonstrate various decay properties for solutions to the above semilinear wave equations on Schwarzschild spacetimes with general large data prescribed on a  Cauchy initial hypersurface.

\subsection{Schwarzschild spacetimes}
Under the above Boyer-Lindquist coordinates, it is well known that $\{r=2M\}$ (the event horizon) is coordinate singularity instead of spacetime singularity, which could be seen from the  Regge-Wheeler tortoise coordinate
\begin{align*}
  &r^*=r+2M\ln(r-2M)-3M-2M\ln M,\quad dr^*=(1-\mu)^{-1}dr,
\end{align*}
or the advanced and retarded Eddington-Finkelstein coordinate system
 \[
 v=\frac{t+r^*}{2},\quad u=\frac{t-r^*}{2}.
 \]
 Then the metric $g$ can also be written as
\begin{align*}
  g=-4(1-\mu)dv^{2}+4dvdr+r^2d\omega^2=-4(1-\mu)du dv+r^2 d\om^2.
\end{align*}
 We will mainly study the solutions to the nonlinear wave equation in the domain of outer communication $\{r\geq 2M\}$, where the level surface $\{v=const\}$ is null under the coordinate system $(v, r, \omega)$. This will cause some degeneracy when doing energy estimates. To obtain a uniform timelike vector field, we will also use the coordinate system  constructed in the work \cite{Tataru10:Strich:Schwarz} by  Marzuola-Metcalfe-Tataru-Tohaneanu. More precisely, let $\lambda(r)$ be an increasing function of $r$ such that
 \[
 2-(1-\mu)\lambda'(r)>0,\quad \lambda(r)\geq r^*
 \]
 and $\lambda(r)=r^*$ when $r\geq \frac{5}{2}M$. Then define
$$\tilde{v}=2v-\lambda(r).$$
By the choice of $\lambda(r)$, we see that away from the event  horizon $r\geq \frac{5}{2}M$ it holds that $\tilde{v}=t$. Under this new coordinate system $(\tv, r, \omega)$, the metric reads as
 \begin{align*}
  g=&-(1-\mu)d\widetilde{v}^2+2\Big(1-(1-\mu)\lambda'(r)\Big)d\widetilde{v}dr
  +\Big(2\lambda'(r)-(1-\mu)\lambda'(r)^2\Big)dr^2+r^2d\omega^2.
\end{align*}
Obviously the vector field $\pa_{\tv}$ is Killing and it can be checked that the level surface $\{\tv=const\}$ is spacelike when $r\geq 2M$. The volume element under these coordinate systems is of the form
\begin{align*}
  d\text{vol}=r^2dr d\widetilde{v} d\omega=r^2dr dt d\omega=r^2(1-\mu)dr^{*}d\omega dt
  =2r^2(1-\mu)dvdud\omega.
\end{align*}

\subsection{Foliations and Energy Flux}
We will study the asymptotic behavior of the solutions in the  domain of outer communication, including the event horizon. Mainly two types of foliations will be frequently used in this paper. Near the event horizon, the hypersurfaces with constant $ \widetilde{v}$ used in  \cite{Tataru10:Strich:Schwarz} would be of particular importance while in the region near the future null infinity, we will switch to the double null foliation.

Let  $r_0<2M$ be some positive constant (slightly smaller than $2M$). We will study the solution in the region
\begin{equation*}
  \mathcal{M}=\{\widetilde{v}\geq 0,\quad r\geq r_0\},
\end{equation*}
which can be foliated with constant-$\widetilde{v}$ hypersurfaces $\widetilde{\Sigma}_{\widetilde{v}}$ defined by
$$\widetilde{\Sigma}_{\widetilde{v}_0}=\{\widetilde{v}=\widetilde{v}_0\}\cap\{r\geq r_0\}.$$
This foliation gives good energy estimates through $\widetilde{\Sigma}_{\widetilde{v}}$ near the event horizon as $ \partial_r$ is spacelike. In the far away region, we instead use the modified double null foliation. The region bounded by $ \widetilde{\Sigma}_{\widetilde{v}_1}$ and $ \widetilde{\Sigma}_{\widetilde{v}_2}$ will be denoted as $\widetilde{\mathcal{D}}_{\tv_1, \tv_2}$

Let $R>5M$ be a constant to be determined later. Let $ \mathcal{H}_{u}$ be the outgoing null hypersurface
$$ \mathcal{H}_{u}=\{t-r^*=2u,\quad r>R,\quad t\geq 0\}.$$
Then we define
\begin{align*}
\Sigma_u=\{r_0<r<R,\quad \widetilde{v}=2u+R^*,\quad \widetilde{v}\geq 0 \}\cup \mathcal{H}_{u},
\quad  R^{*}=r^*|_{r=R}.
\end{align*}
We may also use a truncated part $\Sigma_u^{v_0}=\Sigma_u\cap\{v\leq v_0\}$. Similarly define the incoming null hypersurface $ \underline{\mathcal{H}}_{v}$
\[
\underline{\mathcal{H}}_{v}=\{ t+r^*=2v,\quad r\geq R,\quad \tv\geq 0\}
\]
as well as the  truncated part $\underline{\mathcal{H}}_v^{u_1,u_2}=\underline{\mathcal{H}}_v\cap\{u_1\leq u\leq u_2\}$. See the figure described below. Similarly $\cD_{u_1, u_2}$ stands for the region bounded by $\Si_{u_1}$ and $\Si_{u_2}$ while $\cD_{u_1, u_2}^{v_0}$ is the region  $\cD_{u_1, u_2}^{v_0}\cap \{v\leq v_0\}$.
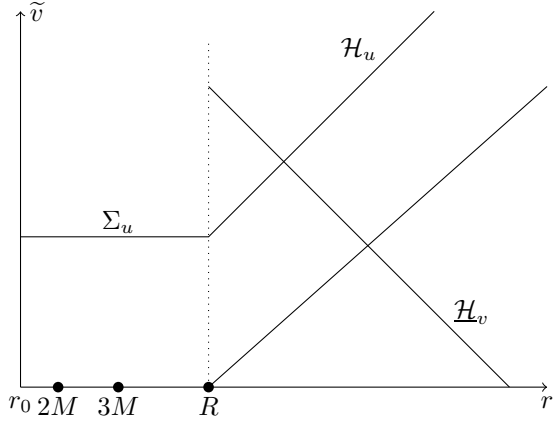
\begin{center}
 \begin{tikzpicture}
  \draw[->] (0,0) node[below]{$r_0$}-- (7,0) coordinate[label = {below:$r$}] (xmax);
  \draw[->] (0,0) -- (0,5) coordinate[label = {right:$\tv$}] (ymax);
     \fill [black] (0.5, 0) node[below]{$2M$} circle (2pt);
    \fill [black] (1.3, 0) node[below]{$3M$} circle (2pt);
  \fill [black] (2.5,0) circle (2pt);
  \draw (2.5,0 ) node[below]{$R$} -- (7,4);
  \draw[dotted] (2.5, 0)--(2.5, 4.6);
  \draw (0, 2)--(2.5, 2)--(5.5, 5);
  \draw  (2.5, 4)--(6.5,0);
     \node at (6, 1) {$\underline{\mathcal{H}}_{v}$};
     \node at (1.3, 2.2) {$ \Si_u$};
     \node at (4.5, 4.5) {$ \mathcal{H}_{u}$};

\end{tikzpicture}
\end{center}

For any hypersurface $\Si$ in $\mathcal{M}$ and a scalar field $\psi$, denote $E[\psi](\Si)$ as the standard energy flux (up to some universal constant) for the scalar field $\psi$ through $\Si$. For $p\geq 0$, let
\[
E[\psi, p](\Si)=E[\psi](\Si)+p\int_{\Si}|\psi|^{p+1}d\sigma.
\]
Here $d\sigma$ is the surface measure on $\Si$. In application, the hypersurfaces will be $\widetilde{\Si}$, $\Si_u$, $\mathcal{H}_u$ or $\underline{\mathcal{H}}_{v}$. Since $\Si_u$ is a combination of $\widetilde{\Si}$ and $\mathcal{H}_u$, for the basic three types of hypersurfaces, the explicit energy flux is given as follows:
\begin{align*}
E[\psi](\widetilde{\Si}_{\tv})=\int_{\widetilde{\Si}_{\tv}}(|\pa \psi|^2  & r^2{ +| \psi|^2}) drd\om,\quad E[\psi](\mathcal{H}_u)=\int_{ \mathcal{H}_u}((|L\psi|^2+|\nabb\psi|^2)  r^2{ +| \psi|^2}) dvd\om, \\
&E[\psi](\Hb_v)=\int_{ \Hb_v}((|\Lb\psi|^2+|\nabb\psi|^2)  r^2{ +| \psi|^2}) dud\om.
\end{align*}
Here $\pa$ is short for the full derivative and $L=\pa_t+\pa_{r^*}$, $\Lb=\pa_t-\pa_{r^*}$ and $\nabb$ is the covariant derivative on the sphere with radius $r$. We remark here that although there is a factor $2(1-\mu)$ in the surface measure on $\H_u$ or $\Hb_v$, it is bounded below as $r\geq R\geq 5M$.

The commutators will be the Killing fields in the set
\[
\Ga=\{\pa_{\tv},\quad \Om_{ij}=x_i\pa_j-x_j \pa_i\},
\]
in which  we set $x=r\omega=(x_1,x_2,x_3)$ in the $ (\tv,r,\omega)$ coordinate.
We may also use $\pa_{\om}$ to be short for the angular momentum vector field $\Om_{ij}$. For integer $k$, $Z^k$ is short for $Z_1 Z_2\ldots Z_k$ with $Z_j\in \Ga$. For the higher order energy, we define
\[
E[Z^{\leq k}\phi](\Si)=\sum\limits_{l\leq k,\ Z\in \Ga}E[Z^l\phi](\Si ).
\]
For simplicity, we also use $\mathcal{E}_{k}$ to denote the size of the initial energies
\begin{align*}
\cE_0=E[\phi, p](\tSi_0),\quad \cE_k=E[Z^{\leq k}\phi](\tSi_0).
\end{align*}

 \subsection{Statement of the main theorem}

We will study the Cauchy problem to the nonlinear wave equation \eqref{maineq} on $\mathcal{M}$ with data $\phi(0, x)=\phi_0$, $\pa_{\tv}(0, x)=\phi_1$ prescribed on the initial hypersurface $\{\tv=0\}$.
For simplicity we will assume that the initial data are smooth and compactly supported in $\{r\leq R\}$. The results here can be easily extended to general data bounded in some weighted energy space.
 For such initial data the solution exists globally on $\mathcal{M}$ (see for example \cite{Shatah:YM:generalManifold} \footnote{The method there can be easily adapted to the defocusing semilinear wave equation studied here due to the conservation of energy. See also the discussion in \cite{Blue03:ILD:Schwarz}}).  The purpose of this paper is to show that the solution exhibits certain quantitative decay properties.
Our main results are:
\begin{thm}
\label{thm:potential:decay}
Consider the Cauchy problem to the defocusing semilinear wave equation \eqref{maineq} on the\\ Schwarzschild spacetime with compactly supported initial data $(\phi_0, \phi_1)$. Then it holds that
\begin{itemize}
\item [(1)]
If $5>p>p_0\approx  1.52$, the solution $\phi$ verifies an integrated local energy decay estimate
\begin{equation*}
\begin{split}
  \iint_{\mathcal{M}}\frac{1}{r^2}|\partial_r\phi|^2
  +\left(1-\frac{3M}{r}\right)^2\left(\frac{1}{r^2}|\partial_{\widetilde{v}}\phi|^2
  +\frac{1}{r}|\slashed\nabla\phi|^2\right)
  +\frac{1}{r^4}|\phi|^2+\frac{p}{r}|\phi|^{p+1}d\vol
  \leq C\cE_0
  \end{split}
\end{equation*}
for some constant $C$ depending only on $p$ and $M$.
\item [(2)]
If  $\frac{1+\sqrt{17}}{2}<p<5$, there exists some $k_0\in(1, 2)$ such that
\begin{equation*}
  \iint_{\mathcal{M} }|\phi|^{k_0(p-1)}d\vol
  \leq C
\end{equation*}
for some constant $C$ depending on $p$, $M$ and the zeroth order energy $\cE_0$.
\item [(3)]
For $\frac{1+\sqrt{17}}{2}<p<5$ and any $1<\gamma<2$, the solution satisfies the following pointwise decay estimates
\begin{equation*}
|\phi|\leq \left\{
\begin{aligned}
&C   \cE_1^{C}\sqrt{\cE_2} \tv^{-\frac{\ga}{2}} ,\qquad\qquad \qquad r_0\leq r\leq 5M ,\\
&C \cE_1^{C}\sqrt{\cE_2}  {(rv)^{-\frac{1}{2}} }(1+|u|)^{-\frac{\gamma-1}{2}} , \quad r> 5M
\end{aligned}
\right.
\end{equation*}
for some constant $C$ depending on $\ga$, $p$, $M$ and $\cE_0$.
\end{itemize}
  \end{thm}
Several remarks are in order.
\begin{Remark}
The lower bound $p_0$ is the root of some polynomial and could be improved by modifying the vector field when deriving the integrated local energy estimate. It would be of great interest whether the above integrated local energy estimate holds for all $p>1$ as it is on the flat Minkowski space.
\end{Remark}

\begin{Remark}
 The compact support assumption on the initial data is merely for simplicity. The argument and results can be extended to general data bounded in some weighted energy space by first studying the solution in the exterior of a forward light cone, where the spacetime is close to the flat Minkowski space.
\end{Remark}

\begin{Remark}
 The solution decays only like $t^{-1+\epsilon}$ near the event horizon, which is far from sharp in view of the celebrated Price's law. In our future work, we will show that Price's law holds for sufficiently large $p$ and general smooth and localized initial data.
\end{Remark}

\subsection{Comments on the Proof}
In this section, we briefly discuss the ideas for the proof. Since we are considering a large data problem, inspired by the case on the flat Minkowski space (see early works \cite{Pecher82:decay:3d},
\cite{Strauss:NLW:decay}  and recent advances in \cite{yang:scattering:NLW}, \cite{yang:NLW:ptdecay:3D}), the asymptotic dynamics of the solution relies on  a type of global dispersive estimate for the nonlinearity. The conformal symmetry of the Minkowski space allows one to show the decay in time for the potential energy
\begin{align*}
\int_{\mathbb{R}^3} |\phi(t, x)|^{p+1}dx\leq C (1+t)^{ \max(4-2p,-2)},
\end{align*}
which, combined with the fundamental solution for the linear wave operator, is sufficient to conclude the long  time dynamics of the nonlinear solutions.

New difficulties and challenges arise for deriving such a priori estimate for the solution without losing any derivative on Schwarzschild spacetimes for the existence of trapped null rays at the photon sphere $\{r=3M\} $ (see for example \cite{Ralston69:ILE}, \cite{Jan15:GauB}). However since the Schwarszchild spacetime is asymptotically flat, the main challenge lies in understanding the solution in the region $\{r\leq R_0\}$ for some large constant $R_0>3M$.
The dispersion of the solution in this bounded region can be captured by the integrated local energy, which was first obtained by Morawetz in \cite{Morawetz68:KG} for the study of decay properties for linear solutions. The vector field method of Klainerman, which is the generalization of the original multiplier method of Morawetz, is a powerful and robust tool to obtain such integrated local energy for solutions of linear wave equations on various backgrounds.

In the Schwarzschild context considered here, such Morawetz estimate dates back to the works \cite{Soffer99:NLS:Schwarz},
\cite{Blue03:ILD:Schwarz} and various different approaches have been developed to obtain such estimates for linear solutions in
 \cite{Blue16:Max:Schwarz},
 \cite{Blue20:linGra:Schwarz},
 \cite{Blue08:Max:Schwarz},
 \cite{Blue05:spin2:Schwarz},
\cite{Igor19:linStab:Schwarz},
\cite{Igor07:note:Schwarz},
\cite{dr3},
\cite{Tataru10:Strich:Schwarz}.
Also see generalizations and extensions to the more complicate  family of Kerr black hole spacetimes in
\cite{Blue:hiddensymmetry},
\cite{Blue:energybound:smallKerr},
\cite{Igor19:decay:Teuko:smKerr},
\cite{Igor11:bd:smallKerr},
\cite{Igor:wave:kerr:ful},
\cite{Mihai20:ILE:Kerr:p},
\cite{Tataru15:ILE:Maxwell:Schwarz},
\cite{Mihai11:ILE:Kerr},
\cite{Mihai12:Strichartz:Kerr}
and to the asymptotic flat spacetimes in
\cite{Blue13:decay:trap},
\cite{Tataru20:ILE:nontrap},
\cite{Tataru12:Strichartz:par},
\cite{Tataru13:localdecay}. However among all these different methods, the spherical harmonic decomposition or Fourier analysis may not be such successful for the large data nonlinear problem studied here. We hence turn to the multiplier method used in \cite{Tataru10:Strich:Schwarz}, in which they proved the boundedness of the nonnegative spacetime integral by explicitly constructing a multiplier of the form $f(r^*)\partial_{r^*}$. Directly applying this multiplier to the nonlinear solution considered here, the coefficient of the potential energy $|\phi|^{p+1}$ will be positive as long as the power $p$ for the nonlinearity has the lower bound given in the main theorem. The integrated local energy estimate for the nonlinear solution then follows in a similar way as that in \cite{Tataru10:Strich:Schwarz}.

It is possible to improve the lower bound for $p$ by carefully modifying the multiplier used in \cite{Tataru10:Strich:Schwarz}. However it may be of particular interest to ask whether the integrated local energy estimate in the main theorem holds for all $1<p<5$ as that in the Minkowski space \cite{yang:scattering:NLW}. At this point, we would like to compare the model problem considered here to the nonlinear equation
\begin{equation}
\label{eq:NW:Blue}
\Box_g \phi= (1-\frac{2M}{r})^{\frac{p-3}{2}}|\phi|^{p-1}\phi
\end{equation}
 investigated by  Blue-Soffer in \cite{Blue07:Mor:NLW:Sch}, where they derived the integrated local energy  estimate for all $1<p\leq 3$. The existence of the multiplier they used was verified numerically. We remark here that the explicit vector field constructed in \cite{Tataru10:Strich:Schwarz} also applies to all $p>1$ for the above nonlinearity. See detailed discussion in section \ref{sec:ILE}.
Note  that the nonlinearity  $(1-\frac{2M}{r})^{\frac{p-3}{2}}|\phi|^{p-1}\phi$ is more singular near the event horizon. This leads us to believe that the integrated local energy decay estimate should hold for all $1<p<5$ for solutions to the equation \eqref{maineq} considered here.

Now in order to obtain pointwise decay estimates for the solutions, we need higher order energy estimates for the solution, which then requires improved dispersive bound on the nonlinearity. Inspired by the related results on the flat Minkowski spacetime, we make use of the robust $r$-weighted energy estimates original introduced by Dafermos-Rodnianski in \cite{newapp} (see for example the general framework discussed in \cite{Georgios16:rp}, \cite{Volker13:LSchw}). This approach gives a hierarchy of $r$-weighted energy estimates derived by using the vector field $r^{\gamma}(\partial_t+\partial_{r^*})$ as multipliers. Combined with the integrated local energy estimates, a suitable choice of $\gamma$ (depending on the power $p$) leads to a spacetime bound
  \begin{equation*}
   \||\phi|^{p-1}\|_{L_{\widetilde{v}}^{k}L_x^{l}(\mathcal{M})}\lesssim 1.
\end{equation*}
for some pair $(k, l)$  such that $ \frac{1}{k}+\frac{3}{l}=2$. See details in Lemma \ref{k22}.  This bound together with the Strichartz estimates for linear solutions on Schwarzschild spacetime in \cite{Tataru10:Strich:Schwarz} is sufficient to conclude the boundedness of the first order energy of the solution
\begin{equation*}
 E[Z\phi](\Sigma_u)+\|Z\phi\|_{L_{\widetilde{v}}^{p_1}L_x^{q_1}(\cD_{u})} \lesssim   1
\end{equation*}
for vector field $Z=\partial_t$ or the angular momentum $\Omega_{ij}$. This higher order energy bound helps to control the solution near the photon sphere in view of the degeneracy of the integrated local energy estimate for the solution.
 A pigeonhole argument as proposed in \cite{newapp} then leads to the energy flux decay of the solution
  \begin{align*}
  E[\phi](\Sigma_u)\lesssim
  u_+^{-\gamma_0}
\end{align*}
for some $\gamma_0>1$.

By using standard argument via Sobolev embedding, the pointwise decay estimate then follows once we have energy flux decay for the first order derivatives of the solution $Z\phi$. Again the trapping problem near $r=3M$ can be solved by a rough bound on the higher order derivatives of the solution, that is $Z^2\phi$. There are two types of nonlinear terms after commuting the equation with $Z^2$
\begin{equation*}
\Box_g (Z^2 \phi)=p(p-1)|\phi|^{p-3}\phi Z\phi Z\phi+p|\phi|^{p-1}Z^2\phi.
\end{equation*}
The second type of nonlinear term is of the same structure as that for the equation for $Z\phi$. And hence can be bounded similarly. For the first type of lower order nonlinearity, it seems that the Strichartz estimate and energy bound for $Z\phi$ are not sufficient to obtain a uniform bound. However by interpolation, we are able to show that the $L_{\tv}^1 L_x^2$ norm of $|\phi|^{p-2}| Z\phi|^2$ grows at most polynomially, that is
\begin{align*}
\||\phi|^{p-2}|Z\phi|^2 \|_{L_{\tv}^1L_x^2(\tv_0\leq \tv \leq  \tv_1)} \les  \tv_1,
\end{align*}
which then indicates that the higher order energy as well as the spacetime mixed norm
\begin{equation*}
  E[Z^2\phi](\tSi_{\tv_1}) + \|Z^2\phi\|_{L_{\tv}^{p_1}L_x^{q_1}( \tv_0\leq \tv \leq \tv_1)}^2
  \lesssim \tv_1
\end{equation*}
also grows in time. This mind growth is allowed in view of the following technical lemma
 \begin{equation*}
  \|f\|_{L^2(R^3)}
  \les \ln t\||\ln{ ||x|-1|}|^{-1}f\|_{L^2(R^3)}+t^{-N}\|f\|_{L^q(R^3)},\quad N>0, q>2,\quad  t>2
  \end{equation*}
  for all function $f$ supported in $ \{\frac{1}{2}\leq|x|\leq \frac{3}{2}\}$. To apply the pigeonhole argument to $Z\phi$, in view of the trapping near the photon sphere $r=3M$, one needs to control $$\int_{u_1}^{u_2}\int_{ r\approx 3M}|(\partial_{\tv}, \slashed\nabla)Z\phi|^2 d\vol,$$ which with the help of the above inequality by letting $N$ large enough and $q=12$ can be bounded above by
\begin{align*}
 \int_{u_1}^{u_2}\int_{ |r^*|\leq \f12 }|\ln u_{+}|^2|\ln |r^{*}||^{-2}|(\partial_{\tv},\slashed\nabla)Z\phi|^2 dx
 +\int_{u_1}^{u_2} u_{+}^{-N} \|( \partial_{\tv},\slashed\nabla) Z\phi\|_{L_x^{12}( |r^*|\leq \f12)}^2du.
\end{align*}
The second term on the right hand side can be further controlled by using the higher order energy estimate for $\partial_t Z\phi$ while the first term can be bounded by the energy flux of $Z\phi$ in view of the log-loss integrated local energy estimate in \cite{Tataru10:Strich:Schwarz}. See Proposition \ref{prop:IS:v} for more details. Therefore we can derive that
\begin{equation*}
  \int_{u_1}^{u_2}E[Z\phi](\Sigma_u)du
     \les  (\ln (u_2)_+)^2 E[Z\phi](\Si_{u_1})+ \int_{\mathcal{H}_{u_1}}
  r|L(rZ\phi)|^2dvd\omega,
\end{equation*}
which combined with the standard energy estimate is sufficient to conclude the energy flux decay for $Z\phi$
  $$E[Z\phi](\Sigma_{u})\lesssim u_{+}^{-\gamma}$$
for all $1<\gamma<\gamma_0$. Once we have this higher order energy flux decay, the pointwise decay estimate for the solution then follows by using standard argument.

We make a convention through out this paper that $A\les B$ means there is a constant $C$ depending only on the power $p$ of the nonlinearity, the constant $\ga\in (1, 2)$, the mass $M$, the initial energy $\cE_0=E[\phi, p](\tSi_0)$ (for general unrestricted data, it should be some weighted energy) and some small positive constant $\epsilon$ such that $A\leq CB$. We emphasize here that the implicit constant also relies on  $R$ used to foliate the spacetimes. However $R$ will be chosen to rely  only on $p$, $M$ and $\ga$.

\subsection*{Acknowledgment}
D. Wei is supported by the National Key R\&D Program of China 2021YFA1001500.
S. Yang is supported by the National Key R\&D Program of China 2021YFA1001700 and the National Science Foundation of China  12171011,  12141102.

\section{Vector field method and Strichartz estimates}
We review the vector field  method for wave equations. For a scalar field $\psi$, define the associated energy momentum tensor
\begin{equation*}
\label{Tk}
  T_{\alpha\beta}[\psi, k]=\partial_{\alpha}\psi\partial_{\beta}\psi
  -\frac{1}{2}g_{\alpha\beta}\Big(\partial^{\gamma}\psi\partial_{\gamma}\psi
  +\frac{2k}{p+1}|\psi|^{p+1}\Big), \quad k=0, 1.
\end{equation*}
We may also use $T_{\a\b}[\psi]$ be short for $T_{\a\b}[\psi, 0]$. The case when $k=1$ will only be used for the  zeroth order energy estimates.
For any vector field $X$, a function $q$ and a 1-form $m$, define the current
\begin{equation*}
  P_{\alpha}[\psi,X,q,m, k]
  =T_{\alpha\beta}[\psi, k]X^{\beta}+q\psi\partial_{\alpha}\psi
  -\frac{1}{2}\partial_{\alpha}q\psi^2+\frac{1}{2}m_{\alpha}\psi^2.
\end{equation*}
Then we can compute the divergence
\begin{equation*}
  \nabla^{\alpha} P_{\alpha}[\psi,X,q,m,k]=\Box_g\psi (X\psi+q\psi)-\frac{k}{p+1} \div (X  |\psi|^{p+1})
  +Q[\psi,X,q,m],
\end{equation*}
in which
\begin{equation*}
  Q[\psi,X,q,m]
  =T_{\alpha\beta}[\psi]\pi_X^{\alpha\beta}
  +q\partial^{\alpha}\psi\partial_{\alpha}\psi+\psi m_{\alpha}\partial^{\alpha}\psi
  +\frac{1}{2}(\nabla^{\alpha}m_{\alpha}-\Box_g q)\psi^2.
\end{equation*}
Here $\pi_X=\frac{1}{2}\mathcal{L}_Xg$ is the deformation tensor of the Schwarzschild metric $g$ along the vector
field $X$. By Stokes' formula we then obtain the energy identity
\begin{equation}
\label{eq:energy:id}
    \int_{\partial\mathcal{D}}i_{P[\psi,X,q,m,k]}d\text{vol}
    =\iint_{\mathcal{D}}  \Box_g\psi (X\psi+q\psi)-\frac{k}{p+1} \div (X  |\psi|^{p+1} )
  +Q[\psi,X,q,m]d\text{vol}
\end{equation}
for any domain $\mathcal{D}$ in $\mathcal{M}$.

Next we give two lemmas which will be used frequently in this paper.
\begin{Lem}
\label{main1}
Let $f(t)$, $a(t)$ be continuous functions defined on $[0,\infty)$. For some $p_1, p_2\geq 1$, assume $a\in L^{p_1}$ and $f(t)$ verifies
\begin{equation*}
  \|f\|_{L^{p_2}_{([t_1,t_2])}}\leq C_0+C_0\|a\|_{L^{p_1}_{([t_1,t_2])}}
  \|f\|^{1+\theta}_{L^{p_2}_{([t_1,t_2])}}
  \end{equation*}
for some positive constants $C_0, \theta$ and any $t_2\geq t_1\geq 0$.
Then there exists a constant $C_1$ depending only on $\|a\|_{L^{p_1}}$ and $\theta$ such that
$$\|f\|_{L^{p_2}([0,\infty))}\leq C_1 C_0.$$
\end{Lem}
\begin{proof}
For a small positive constant $\delta$, define
$$g(s)=C_0+C_0\delta s^{1+\theta}-s.$$
Then $g$ satisfies
$$g(0)>0,\quad g(\infty)>0,\quad g'(s)=C_0 \delta(1+\theta)s^{\theta}-1.$$
Thus $g$ has exactly two distinct positive zeros if  $\delta$ is sufficiently small (to ensure that $g(s)$ takes negative values).

Since $a\in L^{p_1}$, we can choose a finite sequence $\{t_i\}_{i=0}^{n}$ such that
$$t_0=0,\quad t_i<t_{i+1},\quad t_n=\infty,\quad \|a(t)\|_{L^{p_1}_{([t_i,t_{i+1}])}}\leq \delta,\quad n\leq \delta^{-1} \|a\|_{L^{p_1}}+1.$$
By the assumption on $f$, for any $t\in[t_i,t_{i+1}]$ we have
$$\|f(t)\|_{L^{p_2}_{([t_i,t])} }\leq C_0+C_0\delta
\|f(t)\|^{1+\theta}_{L^{p_2}_{([t_i,t])}},$$
from which we conclude that $\|f(t,x)\|_{L^{p_1}_{([t_i,t])}}$ can not exceed the first zero of $g(s)=0$. Hence it is bounded. Now the conclusion of the Lemma follows by summing up finite times.
\end{proof}

Based on this lemma, we prove another refined version.
\begin{Lem}
\label{main}
Let $G(t)\geq 0$, $f(t,x)$, $a(t,x)$ be continuous functions with $t\in [0, \infty)$, $x\in \mathbb{R}^3$. Assume that $a\in L^{k}_{t}L_x^{l}$ for some pair $(k, l)$ such that  $\frac{1}{k}+\frac{3}{l}=2$, $1<k<2<l<3$. Let $(k', l')$ be the dual pair $\frac{1}{k}+\frac{1}{k'}=1$, $\frac{1}{l}+\frac{1}{l'}=\frac{1}{2}$.
  If for some  nonnegative constants $C_0$, $\delta$, $q$
\begin{equation*}
  \sup_{t'\in[0,t]}G(t')+\|f\|_{L^{k'}_{([0,t])}L_x^{l'}}\leq C_0\Big(G(0)+\delta t^{q}+\|af\|_{L^{1}_{([0, t])}L_x^{2}}\Big)
\end{equation*}
for all $t\geq 0$. Then there exists a constant $C_1$ depending only on $C_0$ and the norm $\|a\|_{L_t^k L_x^l}$ such that
$$\sup_{t'\in[0,t]}G(t')+\|f\|_{L^{k'}_{([0,t])}L_x^{l'}}\leq C_1(G(0)+\delta t^{q}),\quad \forall t\geq 0.$$

\end{Lem}
\begin{proof}
 Similar to the proof of the previous lemma, we can find a finite sequence  $\{t_i\}_{i=0}^{n}$ such that
$$t_0=0,\quad t_i<t_{i+1},\quad t_n=\infty,\quad \|a(t,x)\|_{L^{k}_{([t_i,t_{i+1}])}L_x^{l}}\leq \frac{1}{2C_0}. $$
Here without loss of generality, we may assume $C_0\neq 0$.
We prove by induction that
\begin{align}
\label{eq:induct}
\sup_{t'\in[0,t_k]}G(t')+\|f\|_{L^{k'}_{([0,t_k])}L_x^{l'}}\leq D_k(G(0)+\delta t_k^{q}),\quad 0\leq k\leq n
\end{align}
For some constant $D_k$ relying only on the norm $\|a\|_{L_t^k L_x^l}$ and $C_0$.
The case when $k=0$ is trivial. Assume that the above estimate holds for $k-1$. In view of the assumption, we derive that
\begin{align*}
&\sup_{t'\in[0, t_k]}G(t')+\|f\|_{L^{k'}_{([t_{k-1},t_k])}L_x^{l'}}\\
&\leq C_0(G(0)+\delta t_k^{q}+\|af\|_{L^{1}_{([0, t])}L_x^{2}})\\
&\leq C_0(G(0)+\delta t_k^{q}+\|a\|_{L_t^{k}L_x^{l}} \|f\|_{L_t^{k'}([0, t_{k-1}])L_x^{l'}}+ \|a\|_{L_t^{k}([t_{k-1}, t_k])L_x^{l}} \|f\|_{L_t^{k'}([ t_{k-1}, t_k])L_x^{l'}})\\
&\leq \frac{1}{2}\|f\|_{L_t^{k'}([ t_{k-1}, t_k])L_x^{l'}}+C_0G(0)+C_0 \delta t_k^q+\|a\|_{L_t^k L_x^l} D_{k-1}(G(0)+\delta t_{k-1}^{q}),
\end{align*}
from which we conclude that \eqref{eq:induct} holds for $k$ by choosing
\begin{align*}
D_k=2(C_0+\|a\|_{L_t^k L_x^l}D_{k-1}) +D_{k-1}.
\end{align*}
Since $n$ relies only on $\|a\|_{L_t^k L_x^l}$, the lemma then follows by noting that \eqref{eq:induct} also holds when $t_n$ is replaced with $t>t_{n-1}$.
\end{proof}

Our method to control the nonlinearity heavily relies on the global Strichartz estimates obtained by
Marzuola-Metcalfe-Tataru-Tohaneanu in \cite{Tataru10:Strich:Schwarz}. We call the pair  $(p_1, q_1)$ is admissible if
\[
\frac{1}{p_1}+\frac{3}{q_1}=\frac{1}{2},\quad 6\leq q_1 <\infty.
\]
For a region $\cD$ and a scalar field $\psi $ on $\mathcal{M}$, define the norm
\begin{align*}
\|\psi \|_{LE(\cD)}=\iint_{\cD}
  \frac{1}{r^2}|\partial_r\psi|^2
  +\Big(1-\ln\Big|1-\frac{3M}{r}\Big|\Big)^{-2}
  \left(\frac{1}{r^2}|\partial_{\widetilde{v}}\psi|^2
  +\frac{1}{r}|\slashed\nabla\psi|^2\right)
  +\frac{1}{r^4}|\psi|^2\dvol.
\end{align*}
This is the log-loss integrated local energy.
We recall the log-loss integrated local energy estimate and Strichartz estimate proved in \cite{Tataru10:Strich:Schwarz} for linear waves on Schwarzschild spacetimes. Let
$$\tilde{\Ga}_{\tv_1, \tv_2}=\{r=r_0\}\cap \{\tv_1\leq \tv\leq \tv_2\}. $$
\begin{Prop}
\label{prop:IS:v}
 Let $(p_1, q_1)$ be admissible. For all $\tv_2\geq \tv_1$, it holds that
\begin{equation*}
    \sup_{\widetilde{v}\in [\widetilde{v}_1,\widetilde{v}_2]}E[\psi](\widetilde{\Sigma}_{\widetilde{v}}\cup \tilde{\Ga}_{\tv_1, \tv})
    +\|\psi\|_{LE(\tcD_{\tv_1, \tv_2})}+\|\psi\|^2_{L_{\tv}^{p_1}L_{x}^{q_1}(\tcD_{\tv_1, \tv_2})} \lesssim E[\psi](\widetilde{\Sigma}_{\widetilde{v}_1})
    +\|\Box_g\psi\|^2_{L_{\tv}^1L^2_x(\tcD_{\tv_1, \tv_2})}.
\end{equation*}
\end{Prop}
\begin{proof}
The case when $\Box_{g}\psi=0$ is given by Theorem 1.3 and Theorem 1.4 in \cite{Tataru10:Strich:Schwarz}  (the integral of $ \psi^2$ follows from Hardy's inequality). The above local version then follows by using Duhamel's formula together with the approximate paramatrix constructed in Proposition 4.12 in  \cite{Tataru10:Strich:Schwarz}.
\end{proof}
Now we need to adapt the above log-loss integrated local energy estimate and Strichartz estimate to the smaller region bounded by $\widetilde{\Si}_{\tv_{1}}$ and $\Si_{u_0}$
$$\mathcal{D}_{u_0}^{\tilde{v}_1}=\{r\geq r_0,u\geq u_0\geq -\f12 R^*,\quad \widetilde{v}_0{ =2u_0+R^*}<\widetilde{v}<\widetilde{v}_1\}$$
\begin{Prop}
\label{L1L2foru}
Assume that $(p_1, q_1)$ is admissible. Then
  \begin{equation*}
    \sup_{u\geq u_0}E[\psi](\Sigma_{u}\cap \mathcal{D}_{u_0}^{\tv_1})
    +\|\psi\|_{LE(\mathcal{D}_{u_0}^{\tv_1})}^2
    +\|\psi\|^2_{L_{\widetilde{v}}^{p_1}L_{x}^{q_1}(\mathcal{D}_{u_0}^{\tv_1})} \lesssim E[\psi](\Sigma_{u_0})+\|\Box_g\psi\|^2_{L^1_{\widetilde{v}}L^2_x(\mathcal{D}_{u_0}^{\tv_1})}.
\end{equation*}
\end{Prop}
\begin{proof}
The bound for $E[\psi](\Sigma_{u}\cap \mathcal{D}_{u_0}^{\tv_1})$
 is the standard energy estimate 
 . It remains to adapt the above log-loss integrated local energy estimates and Strichartz estimate to region bounded by null hypersurfaces. Split $\psi$ into two parts
\begin{equation*}
  \psi=\psi_1+\psi_2,
\end{equation*}
in which
$\psi_1$, $\psi_2$ solve the linear wave equation
\begin{equation*}
  \Box_g\psi_1=0, \quad \Box_g\psi_2=\chi_{\cD_{u_0}^{\tv_1}}\Box_g\psi
\end{equation*}
such that
$\psi_2$ vanishes on the initial hypersurface $\widetilde{\Si}_{\tv_0}$.
 Here $\chi_{\cD_{u_0}^{\tv_1}}$ is the characteristic function of the region  $\cD_{u_0}^{\tv_1}$.
In view of finite speed of propagation,  $\psi_2$ vanishes in $(\cD_{u_0}^{\tv_1})^c\cap \widetilde{\cD}_{\tv_0, \tv_1}$, thus $\psi_1$ coincides $\psi$ on $\Si_{u_0}\cap{\cD_{u_0}^{\tv_1}}$ (up to derivatives if necessary). By using the previous Proposition \ref{prop:IS:v}, we in particular conclude that
\begin{equation*}
\|\psi_2\|_{LE(\mathcal{D}_{u_0}^{\tv_1})}^2
+\|\psi_2\|^2_{L_{\widetilde{v}}^{p_1}L_{x}^{q_1}(\mathcal{D}_{u_0}^{\tv_1})} \lesssim \|\Box_g\psi\|^2_{L^1_{\widetilde{v}}L^2_x(\mathcal{D}_{u_0}^{\tv_1})},
\end{equation*}
which in particularly implies that the Proposition is reduced to controlling  $\psi_1$.

Consider another linear wave $\psi^{\widetilde{v}_1}$ such that
$$\Box_g\psi^{\widetilde{v}_1}=0$$
and $\psi^{\tilde{v}_1}$ coincide with $\psi_1$ on $\cD_{u_0}^{\tv_1}\cap (\widetilde{\Si}_{\tv_1}\cup \{r=r_0\})$
(that is the boundary $\mathcal{B}_4$ in the following figure).
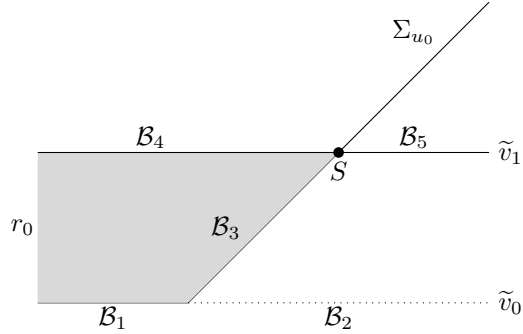
\begin{center}
 \begin{tikzpicture}
  \draw (0,0) -- (2,0)--(6, 4);
  \fill[gray!30] (0, 0)--(2, 0)--(4, 2)--(0, 2)--cycle;
  \draw (0,2) -- (6,2) node[right]{$\tv_1$};
  \fill [black] (4, 2) node[below]{$S$} circle (2pt);
 \draw[dotted] (2, 0)--(6, 0) node[right]{$\tv_0$};
 \node at (-0.2, 1) {$r_0$};
 \node at (1.5, 2.2) {$\mathcal{B}_4$};
 \node at (5, 3.6) {$\Si_{u_0}$};
 \node at (2.5, 1) {$\mathcal{B}_3$};
  \node at (1, -0.2) {$\mathcal{B}_1$};
  \node at (4, -0.2) {$\mathcal{B}_2$};
   \node at (5, 2.2) {$\mathcal{B}_5$};
\end{tikzpicture}
\end{center}

On the exterior boundary $\mathcal{B}_5= \widetilde{\Si}_{\tv_1}\backslash \mathcal{B}_{4}$, we extend the data for $\psi^{\widetilde{v}_1}$ such that
$$E[\psi^{\widetilde{v}_1}](\widetilde{\Sigma}_{\widetilde{v}_1})\leq 10 E[\psi^{\widetilde{v}_1}](\mathcal{B}_4).$$
By solving the equation backward and in view of the above integrated local energy estimate and Strichartz estimate of Proposition \ref{prop:IS:v}, we derive that
\begin{equation*}
  \sup_{u\geq u_0}E[\psi^{\widetilde{v}_1}](\Sigma_{u}\cap \tcD_{ \tv_0, \tv_1})+\|\psi^{\widetilde{v}_1}\|_{LE( \tcD_{\tv_0, \tv_1})}^2
+\|\psi^{\widetilde{v}_1}\|^2_{L_{\widetilde{v}}^{p_1}L_{x}^{q_1}
(\tcD_{\tv_0, \tv_1})} \lesssim
E[\psi^{\widetilde{v}_1}](\widetilde{\Sigma}_{\widetilde{v}_0}).
\end{equation*}
By the above extension and standard energy estimate we have
\begin{align*}
&E[\psi^{\widetilde{v}_1}](\widetilde{\Sigma}_{\widetilde{v}_0})\les E[\psi^{\widetilde{v}_1}](\mathcal{B}_1\cup \mathcal{B}_3)+E[\psi^{\widetilde{v}_1}](\mathcal{B}_5)\les E[\psi^{\widetilde{v}_1}](\mathcal{B}_1\cup \mathcal{B}_3)+E[\psi^{\widetilde{v}_1}](\mathcal{B}_4)\\ \les& E[\psi^{\widetilde{v}_1}](\mathcal{B}_1\cup \mathcal{B}_3)=E[\psi_1](\Si_{u_0}\cap{\cD_{u_0}^{\tv_1}})\leq E[\psi](\Si_{u_0}).
\end{align*}
Here we used the finite speed of propagation to conclude that $\psi^{\widetilde{v}_1}=\psi_1$ in $\cD_{u_0}^{\tv_1}$, thus
\begin{equation*}
  \sup_{u\geq u_0}E[\psi_1](\Sigma_{u}\cap \mathcal{D}_{u_0}^{\tv_1})+\|\psi_1\|_{LE(\cD_{u_0}^{\tv_1})}^2
+\|\psi_1\|^2_{L_{\widetilde{v}}^{p_1}L_{x}^{q_1}
(\cD_{u_0}^{\tv_1})} \lesssim
E[\psi](\Sigma_{u_0}).
\end{equation*}
The Proposition then follows.
\end{proof}
By letting $\widetilde{v}_1\to+\infty$, we can replace the region $\mathcal{D}_{u_0}^{\tv_1}$ by the future of $\widetilde{\Si}_{\tv_0}$ or $\Si_{u_0}$. Moreover the above Proposition together with Lemma \ref{main1} and Lemma \ref{main} leads to the following Corollary.
\begin{Cor}\label{L1L2}
Let $(p_1, q_1)$ be admissible and $(k, l)$ be the dual pair $\frac{1}{k}+\frac{3}{l}=2,$ $1<k<2<l<3$. Consider the linear wave equation $\Box_g\psi=a\psi$ with coefficient $a \in L_{\widetilde{v}}^{k} L_{x}^{l}(\mathcal{M})$. Then we have
  \begin{equation*}
  \begin{split}
    &\sup_{u\geq u_0}E[\psi](\Sigma_{u})
    +\|\psi\|_{LE(\mathcal{D}_{u_0})}^2
    +\|\psi\|^2_{L_{\widetilde{v}}^{p_1}L_{x}^{q_1}(\cD_{u_0})}
    \lesssim E[\psi](\Sigma_{u_0}),
    \\&
    \sup_{\widetilde{v}\geq \widetilde{v}_0}E[\psi](\widetilde{\Sigma}_{\widetilde{v}})
    +\|\psi\|_{LE(\tcD_{\tv_0})}^2
    +\|\psi\|^2_{L_{\widetilde{v}}^{p_1}L_{x}^{q_1}(\tcD_{\tv_0})}
    \lesssim E[\psi](\widetilde{\Sigma}_{\widetilde{v}_0}).
    \end{split}
\end{equation*}
Here $\cD_{u_0}$ is short for $\cD_{u_0}^{\infty}$ (similarly for $\tcD_{\tv_0}$) and the implicit constant also relies on the norm $\|a\|_{L_{\widetilde{v}}^{k} L_{x}^{l}(\mathcal{M})}$
\end{Cor}
\begin{proof}
We only prove the first estimate and the second one follows in a similar way. In view of the above Proposition \ref{L1L2foru}, we derive that
  \begin{equation*}
    \|\psi\|^2_{L_{\widetilde{v}}^{p_1}L_{x}^{q_1}(\cD_{u_0}^{\tv_1})} \lesssim E[\psi](\Sigma_{u_0})+\|a\psi\|^2_{L^1_{\widetilde{v}}L^2_x (\cD_{u_0}^{\tv_1})}
\end{equation*}
for any $\widetilde{v}_1\geq \widetilde{v}_0$.  Since $a\in L_{\widetilde{v}}^{k} L_{x}^{l}(\mathcal{M})$, the Corollary then follows from Lemma \ref{main}.
\end{proof}

\section{Rough decay estimate for the scalar field $\phi$}
We first derive a rough decay estimate for the scalar field. This will be achieved by obtaining an integrated local energy decay estimate for the scalar field together with uniform energy bound for the derivatives of the solution.

\subsection{The integrated local energy decay estimate}
\label{sec:ILE}
Since we are considering large data problem, the asymptotic decay of the solution heavily relies on the decay of the potential energy. Similar to the case in the flat Minkowski space, a weak version of such decay is an integrated local energy estimate for the potential energy. The integrated local energy estimate for linear waves on black  hole spacetimes have been studied extensively in the past decades.

We will adapt the vector field method in  \cite{Tataru10:Strich:Schwarz} by Marzuola-Metcalfe-Tataru-Tohaneanu  for deriving the integrated local energy estimate for  linear waves to the nonlinear model \eqref{maineq} in this paper. Before we recall a key lemma   from \cite{Tataru10:Strich:Schwarz}, let's denote
\begin{equation*}
\begin{split}
\mathcal{W}[\psi,p]&=\frac{1}{r^2}|\partial_r\psi|^2
  +\left(1-\frac{3M}{r}\right)^2\left(\frac{1}{r^2}|\partial_{\widetilde{v}}\psi|^2
  +\frac{1}{r}|\slashed\nabla\psi|^2\right)
  +\frac{1}{r^4}|\psi|^2+\frac{p}{r}|\psi|^{p+1},
  \end{split}
\end{equation*}
We have
\begin{Lem}\label{Ta}
  There exist $K$-invariant smooth spherically symmetric vector field $X$ and 1-form $m(r)$ as well as a function $q(r)$ defined on $r\geq 2M$ such that in $(r,\tilde{v})$ coordinates { (here $K=\partial_{\tv}$)}
  \begin{align*}
  |X|\les 1, \quad |q(r)|\lesssim r^{-1},\quad |q'(r)|\lesssim r^{-2},\quad Q[\psi,X,q,m]\gtrsim \mathcal{W}[\psi, 0],\quad \langle m,dr\rangle |_{r=2M}>0.
  \end{align*}

\end{Lem}
\begin{proof}
See Lemma 2.1 of \cite{Tataru10:Strich:Schwarz}.
\end{proof}
 However, to extend the integrated local energy estimate to the nonlinear model \eqref{maineq}, we sketch the proof for the above lemma.
The vector field $X$ is chosen to be the form
$$X=X_1+\delta X_2 $$
with small positive constant $\delta$ and vector fields  (in $(t,r,\omega)$ coordinates)
\begin{align*}
  &X_1=a(r)(1-\mu)\partial_r,\quad X_2=b(r)(1-\mu)
   (\partial_r-(1-\mu)^{-1}\partial_t ),\quad \mu=\frac{2M}{r}.
\end{align*}
Note that in $(\tv,r,\omega)$ coordinates we have $X_1=a(r)((1-\mu)(\partial_r-\lambda'(r)\partial_{\tv})+\partial_{\tv}),$ $X_2=b(r)(1-\mu)(\partial_r-\lambda'(r)\partial_{\tv}).$

Let $f$ be a smooth function such that $f(x)=x$ on $[-1,\infty]$ and $f(x)=-2$ on $[-\infty,-3]$. Then the function $a(r)$ is chosen to be
 $$a(r)=\frac{1}{r^2\epsilon}f\Big(\epsilon(r-3M)(r+2M)+6\epsilon M^2
\ln\Big(\frac{|r-2M|}{M}\Big)\Big).$$
Here $\epsilon $ is another small positive constant.
Next for the function $b(r)$, let
 $b_0$ be a smooth decreasing function of $r$   such that $b_0(r)=0$ when $r\geq 3M$. Then define the function $b(r)$ to be
$$b=-(1-\mu)^{-1}b_0,\quad r\geq r_0.$$
For the function $q$, let
\begin{align*}
  &t_1(r)= (1-\mu )r^{-2}\partial_r(r^2a(r)),\quad q_0(r)= \frac{1}{2}t_1(r)+\frac{\delta}{r}b(r)(1-\mu),  \quad
  q_1(r)=\chi_{\{r>\frac{5M}{2}\}} r^{-4} (r-3M)^2.
\end{align*}
Here $\chi_{\{r>\frac{5M}{2}\}}$ is a smooth cutoff function which is supported in $\{r>\frac{5M}{2}\}$ and equals $1$ for $r>3M$. Then for some small positive constant $\delta_1<\delta$ the function $q$ is defined to be  $$q=q_0+\delta_1q_1.$$
Finally we recall here that the 1-form $m$ is  compactly supported in $r\leq 3M$.

Based on the above lemma, we obtain integrated local energy estimates for solutions to \eqref{maineq}.
\begin{Prop}
\label{mainenergy}
	For  $p>p_0$($\approx 1.52$), we have the following integrated local energy estimates for the solution $\phi$ to \eqref{maineq}
	\begin{align}
	\label{eq:ILE:tfoil}
	&E[\phi,p](\widetilde{\Sigma}_{\widetilde{v}_1})
	+\iint_{\tcD_{\tv_0, \tv_1}}\mathcal{W}[\phi,p] d\vol \lesssim
    E[\phi,p](\widetilde{\Sigma}_{\widetilde{v}_0}),
    \\&
    \label{eq:ILE:doubnull}
	E[\phi,p](\Sigma_{u_1})
	+\iint_{\cD_{u_0, u_1}}\mathcal{W}[\phi,p]d\vol \lesssim E[\phi,p](\Sigma_{u_0}),
  	\end{align}
    adapted to the two foliations for all $\widetilde{v}_0\leq\widetilde{v}_1$ and $u_0\leq u_1$.
	\end{Prop}
\begin{proof}
Take the vector field  $X$, the function $q$ and the 1-form $m$ as above. For solution $\phi $ to the equation \eqref{maineq}, we can compute that
\begin{equation*}
  \nabla^{\alpha} P_{\alpha}[\phi,X,q,m, 1]=\left(q-\frac{\div (X)}{p+1}\right)|\phi|^{p+1}+Q[\phi,X,q,m].
\end{equation*}
In view of the above  Lemma \ref{Ta}, the main quadratic term $Q[\phi,X,q,m]$ is positive. We next show that the coefficient of the potential energy is also positive, that is,
\begin{align*}
(p+1)q-\div (X)=\frac{p+1}{2}t_1-\div(X_1)+\delta_1 (p+1)q_1-\delta \Big(\div(X_2)-(p+1)r^{-1}b(1-\mu )\Big)>0.
\end{align*}
Now recall that $X=X_1+\delta X_2$. By the definition of the vector field $X_1$, $X_2$, we can compute that
\begin{align*}
  &{\text{div}X_1}=\frac{1}{r^2}\partial_r\big(r^2a(r)
  (1-\mu)\big)=t_1(r)+\frac{2M}{r^2}a(r),\quad {\text{div}X_2}=\frac{1}{r^2}\partial_r\big(r^2b(r)
  (1-\mu)\big).
\end{align*}
Since $b(r)(1-\mu)$  is smooth and supported in $r\leq 3M$, for sufficiently small $\delta$ the contribution from $\delta X_2$ and $b(r)$ is negligible.
For the main term, we have
\begin{align*}
\frac{p+1}{2}t_1-\div(X_1)=\frac{p-1}{2}t_1(r)-\frac{2M}{r^2}a(r).
\end{align*}
By the choice of the function $a(r)$, we see that $a(r)\leq 0$ when $r\leq 3M$ and we can compute the function $t_1(r)$
\begin{align*}
t_1(r)=\left(1-\mu\right)\frac{1}{r^2}\left(2r-M+\frac{6M^2}{r-2M}\right),
\end{align*}
which is smooth and positive away from the horizon. Now consider the function
\begin{align*}
  &\frac{4Ma(r)}{r^2t_1(r)}=\frac{4M((r-3M)(r+2M)+6M^2\ln(r/M-2))}{r((2r-M)(r-2M)+6M^2)}
\end{align*}
on $r\geq 3M$. It is nonnegative and obviously it goes to zero when $r\rightarrow\infty$.
 Hence let
$$p_0=1+\sup\limits_{r\geq 3M}\frac{4Ma(r)}{r^2t_1(r)}\approx 1.52.$$
In particular, when $p>p_0$, we then have the lower bound
\begin{align*}
\frac{p+1}{2}t_1-\div(X_1)=\frac{p-1}{2}t_1(r)-\frac{2M}{r^2}a(r)\gtrsim r^{-1}.
\end{align*}
We divert to briefly demonstrate that for the nonlinearity $(1-\mu)^{\frac{p-3}{2}}|\phi|^{p-1}\phi$ considered by Blue-Soffer in \cite{Blue07:Mor:NLW:Sch} (see the  equation \eqref{eq:NW:Blue}), the leading coefficient for the nonlinearity is positive for all $p>1$. Similarly neglecting the part contributed by the vector field $X_2$ and the function $q_1$ as $\delta$, $\delta_1$ will be chosen to be small, the leading coefficient for $\frac{1}{p+1}(1-\mu)^{\frac{p-3}{2}}|\phi|^{p+1}$ is of the form
\begin{align*}
\frac{p-1}{2}t_1(r)-\frac{2M}{r^2}a(r)- X_1((1-\mu)^{\frac{p-3}{2}})(1-\mu)^{\frac{3-p}{2}}=\frac{p-1}{2}t_1 (1-\frac{2M a}{r^2t_1}),
\end{align*}
which is positive for all $p>1$. In other words, for the nonlinearity $(1-\mu)^{\frac{p-3}{2}}|\phi|^{p-1}\phi$, the integrated local energy estimate holds for all $p>1$.

Recall that the function $q_1$ asymptotically behaves like $r^{-2}$. Hence for sufficiently small $\delta$ and $\delta_1$, the above computations then imply that
\begin{equation*}
  \big(q-\frac{\text{div}X}{p+1}\big)
  |\phi|^{p+1}\gtrsim \frac{1}{r}|\phi|^{p+1}.
\end{equation*}
In the energy identity \eqref{eq:energy:id}, replace the vector field $X$ with $X+CK$ for some large positive constant $C$. Here the Killing vector field $K=\pa_t$. Since $\div(CK)=0$, $\pi_{CK}=0$, the bulk integral is of the form
\begin{align*}
&\Box_g\phi ((X+CK)\phi+q\phi)-\frac{k}{p+1} \div ((X+CK)  |\phi|^{p+1} )
  +Q[\phi,X+CK,q,m]\\
  &=(q-\frac{\div(X)}{p+1})|\phi|^{p+1}+Q[\phi, X, q, m] \gtrsim \mathcal{W}[\phi, p]
\end{align*}
in view of Lemma \ref{Ta} and the above lower bound for the potential energy. It then remains to control the boundary integrals for different regions.
First notice that the vector field 
$$V=r^{-2}[\partial_{\tv}((r-r_0)\phi^2)\partial_{r}-\partial_r((r-r_0)\phi^2)\partial_{\tv}]$$
is divergence free, that is $\div(V)=0$. In particular we have $\int_{\partial D}i_Vd\vol=0$.

For the region $\tcD_{\tv_0, \tv_1}$ bounded by $\tSi_{\tv_0}$, $\tSi_{\tv_1}$ and $\{r=r_0\}$, we compute the integrals under the coordinates $\{\tv, r, \om\}$, under which the Killing vector field $K=\pa_{\tv}$. Recall the volume form
\begin{align*}
d\vol =r^2dr d\widetilde{v} d\omega=2(1-\mu)r^2 dudvd\om,\quad K=\pa_{\tv}.
\end{align*}
For the boundary integral on $\{r=r_0\}$ ($r_0<2M$ is sufficiently close to $2M$), by the definition of $X$, $q$ and $m$, when $r$ close to $r_0$, we have
\begin{align*}
q={-}\delta r^{-1}b_0,\quad X+CK &= CK-2r^{-2}\epsilon^{-1}((1-\mu )(\partial_r-\lambda'(r)\partial_{\tv})+\partial_{\tv}) -\delta b_0 (\partial_r-\lambda'(r)\partial_{\tv}) \\
&=(C+\delta b_0\lambda'(r)-2r^{-2}\epsilon^{-1}(1-(1-\mu )\lambda'(r)))\pa_{\tv}+\big(2r^{-2}\epsilon^{-1}(\mu-1 )-\delta b_0\big)\pa_r.
\end{align*}
 Denote $\tilde{X}=X+CK$.  On $\{r=r_0\}$ we have
\begin{align*}
i_{P[\phi,\tilde{X},q,m,1]}d\text{vol}&=r^2 (P[\phi,\tilde{X},q,m,1])^{r}d\tv d\om=i_{P[\phi,\tilde{X},q,m,1]}d\text{vol}-r^2\tilde{X}^r
\frac{1}{p+1}|\phi |^{p+1}
  d\tv d\om.
\end{align*}
For sufficiently small $\delta>0$ and $r_0<2M$ sufficiently close to $2M$, at $r=r_0$, we see that
\begin{align*}
&\tilde{X}^{\tv}|_{r=r_0} \approx C>0, \quad \tilde{X}^{r}\approx -\delta b_0<0.
\end{align*}
By (2.4) of \cite{Tataru10:Strich:Schwarz} we have\begin{align*}
i_{P[\phi,\tilde{X},q,m,0]}d\text{vol}\gtrsim r^2 ( (\pa_r\phi)^2+(\pa_{\tv}\phi)^2+|\nabb\phi |^2+|\phi |^2 )d\tv d\om.
\end{align*}
Thus 
the above computations imply that
\begin{align*}
i_{P[\phi,\tilde{X},q,m,1]}d\text{vol}\gtrsim r^2 ( (\pa_r\phi)^2+(\pa_{\tv}\phi)^2+|\nabb\phi |^2+|\phi|^{p+1}+|\phi |^2 )d\tv d\om.
\end{align*}
 Next for the integral on $\tSi$, first we have
\begin{align*}
i_{P[\phi,\tilde{X},q,m,1]}d\text{vol}&=-r^2 (P[\phi,\tilde{X},q,m,1])^{\tv}dr d\om=-r^2 (T^{\tv}_{\ \tilde{X} }[\phi, 1]+q\phi\partial^{\tv} \phi
  -\frac{1}{2}\partial^{\tv} q\phi^2+\frac{1}{2}m^{\tv} \phi^2)dr d\om.
\end{align*}
For the quadratic term $T^{\tv}_{\ \tilde{X}}[\phi, 1]$, we compute that
\begin{align*}
 -T^{\tv}_{\ \tilde{X} }[\phi, 1]=-\f12 \tilde{X}^{\tv} g^{\tv\tv} (\pa_{\tv}\phi)^2-(\tilde{X}^{r}g^{r\tv}-\f12 \tilde{X}^{\tv} g^{rr}) (\pa_{r}\phi)^2-\tilde{X}^{r} g^{\tv\tv}\pa_r\phi \pa_{\tv}\phi +\f12 \tilde{X}^{\tv} (|\nabb\phi|^2+\frac{2}{p+1}|\phi |^{p+1}).
\end{align*}
Since the vector field $X$ is smooth and bounded, in particular the component $\tilde{X}^{r}=X^r$ is also uniformly bounded. For the metric components, recall that
\begin{align*}
|g^{r\tv}|\les 1, \quad g^{rr}=1-\mu,\quad g^{\tv\tv}=-\la'(2-(1-\mu)\la')<0
\end{align*}
and $g^{\tv\tv}=-(1-\mu)^{-1}$ when $r\geq \frac{5}{2}M$. Therefore for sufficiently large positive constant $C$, we have
\begin{align*}
C\les - \tilde{X}^{\tv} g^{\tv\tv},\quad |\tilde{X}^{r} g^{\tv\tv}|\les 1
\end{align*}
Near the event horizon, the coefficient of $(\pa_r\phi)^2$ verifies
\begin{align*}
\f12 \tilde{X}^{\tv} g^{rr}-\tilde{X}^{r}g^{r\tv}= \f12 \tilde{X}^{\tv}(1-\mu)+(1-(1-\mu)\la')(\delta b_0+2(1-\mu)r^{-2}\ep^{-1})>0
\end{align*}
for $r$ sufficiently close to $2M$. Here we emphasize that the choice of $r_0$ relies on $\delta$, $M$, $\delta_1$, $C$ and $\epsilon$. Away from the event horizon $r>2M$, we have
\begin{align*}
\f12 \tilde{X}^{\tv} g^{rr}-\tilde{X}^{r}g^{r\tv}\geq \f12 \tilde{X}^{\tv}(1-\mu)-|\tilde{X}^r| |g^{r\tv }|\gtrsim C
\end{align*}
for sufficiently large $C$. Here we may note that the above coefficient is positive when $r\geq 2M$ and sufficiently close to $2M$, which is independent of the choice of $C$.  These computations show that
\begin{align*}
-T^{\tv}_{\ \tilde{X} }[\phi, 1] \gtrsim (1+C(1-\mu))(\pa_r\phi)^2+C ((\pa_{\tv}\phi)^2+|\nabb\phi|^2+|\phi|^{p+1})
\end{align*}
Here $r_0$ is chosen such that $C(r_0^{-1}2M-1)\leq \frac{1}{2}$. Now for the lower order terms, note that the 1-form $m$ is supported in $r\leq 3M$ and $|q|\les r^{-1}$, $|q'|\les r^{-2}$. By using Cauchy-Schwarz's inequality, the lower order terms can be controlled as follows
\begin{align*}
|q\phi\partial^{\tv} \phi
  -\frac{1}{2}\partial^{\tv} q\phi^2+\frac{1}{2}m^{\tv} \phi^2|\les \ep_2^{-1}\frac{|\phi|^2}{r^2}+\ep_2 (|\pa_r\phi|^2{ + |\pa_{\tv}\phi|^2}),\quad \forall 0<\ep_2<1.
\end{align*}
{ For small fixed $\ep_2$ (independent of $C$), the term $\ep_2 (|\pa_r\phi|^2+ |\pa_{\tv}\phi|^2)$ can be absorbed. For the lower order term $r^{-2}|\phi|^2$, we need to use vector field $V=r^{-2}[\partial_r((r-r_0)\phi^2)\partial_{\tv}-\partial_{\tv}((r-r_0)\phi^2)\partial_{r}]$, note that on $\tSi$ we have\begin{align*}
i_{V}d\text{vol}&=\partial_r((r-r_0)\phi^2)dr d\om,\quad \partial_r((r-r_0)\phi^2)=\phi^2+2(r-r_0)\phi\partial_r\phi\geq \phi^2/2-2(r-r_0)^2|\partial_r\phi|^2,
\end{align*}
and $ (r-r_0)^2/r^2\leq(r-r_0)/r=1-\mu+(2M-r_0)/r\leq1-\mu+(2M-r_0)/r_0$. Now we can choose $C'$ large enough (depending on $\ep_2$), then choose $C$ large enough (depending on $C'$), and choose $r_0$ close to $2M$ (depending on $C$) such that} on the hypersurface $\tSi$, it holds that (here $C_1$ depends only on $C'$)
\begin{align*}
&i_{P[\phi,\tilde{X},q,m,1]+C'V}d\text{vol}\\ \gtrsim& \left((1+C(1-\mu)-C_1(r-r_0)^2/r^2)(\pa_r\phi)^2+C ((\pa_{\tv}\phi)^2+|\nabb\phi|^2+|\phi|^{p+1})+|\phi/r|^2\right) r^2dr d\om\\ \gtrsim& \left((1+C(1-\mu))(\pa_r\phi)^2+C ((\pa_{\tv}\phi)^2+|\nabb\phi|^2+|\phi|^{p+1})+|\phi/r|^2\right) r^2dr d\om.
\end{align*}
{ Note that $i_Vd\text{vol}=\partial_{\tv}((r-r_0)\phi^2)d\tv d\om=0$ on $ \{r=r_0\}$}. In particular in view of the energy identity \eqref{eq:energy:id} and $ \int_{\partial D}i_Vd\vol=0$, we obtain the integrated local energy estimate \eqref{eq:ILE:tfoil} for the foliation $\tSi$.
\begin{align*}
i_{P[\phi,\tilde{X},q,m,1]}d\text{vol} \gtrsim \left((1+C(1-\mu))(\pa_r\phi)^2+C ((\pa_{\tv}\phi)^2+|\nabb\phi|^2+|\phi|^{p+1})\right) r^2dr d\om.
\end{align*}
{ Note that $i_Vd\text{vol}=\partial_{\tv}((r-r_0)\phi^2)d\tv d\om=0$ on $ \{r=r_0\}$}. In particular in view of the energy identity \eqref{eq:energy:id} and $ \int_{\partial D}i_Vd\vol=0$,  we obtain the integrated local energy estimate \eqref{eq:ILE:tfoil} for the foliation $\tSi$.

The integrated local energy estimate \eqref{eq:ILE:doubnull} for the foliation $\Si$ is similar. The bulk terms are the same. It suffices to check the boundary integrals on $\H_u$ and $\Hb_v$, for which we move to the double null coordinates system $(u, v, \om)$. On the out going null hypersurface $\H_u$ where $r>R\geq 5M$, first we have
\begin{align*}
m=0,\quad \tilde{X}=C\pa_t+X=\frac{C}{2}(\pa_u+\pa_v)+X.
\end{align*}
Therefore on $\H_u$, we compute that
\begin{align*}
i_{P[\phi,\tilde{X},q,m,1]}d\text{vol}&=-(P[\phi,\tilde{X},q,m,1])^{u} 2(1-\mu)r^2 dvd\om\\
&=-2 (1-\mu)r^2 \left(T^{u}_{\ \tilde{X}}[\phi, 1] +q\phi\partial^{u}\phi
  -\frac{1}{2}\partial^{u}q \phi^2 \right)dvd\om\\
  &=r^2 \left( \tilde{X}^{v}|\pa_v\phi|^2+ (1-\mu)\tilde{X}^{u}  (|\nabb\phi|^2+\frac{2}{p+1}|\phi|^{p+1})+q \phi \pa_v\phi -\f12  \pa_v q \phi^2 \right)dvd\om\\
  &\gtrsim r^2 \left(C(|\pa_v\phi|^2+|\nabb\phi|^2+ |\phi|^{p+1})-r^{-2}|\phi |^2\right) dvd\om,
\end{align*}
{we also have $i_{V}d\text{vol}=\partial_v((r-r_0)\phi^2)dv d\om$ and\begin{align*}
\quad \partial_v((r-r_0)\phi^2)=(1-\mu)\phi^2+2(r-r_0)\phi\partial_v\phi\geq (1-\mu)\phi^2/2-2(r-r_0)^2|\partial_v\phi|^2/(1-\mu),
\end{align*}then (still for $C'$ large enough and $C$ large enough depending on $C'$)\begin{align*}
i_{P[\phi,\tilde{X},q,m,1]+C'V}d\text{vol}&\gtrsim r^2 \left(C(|\pa_v\phi|^2+|\nabb\phi|^2+ |\phi|^{p+1})+r^{-2}|\phi |^2\right) dvd\om,
\end{align*}}
Similarly on the incoming null hypersurface $\Hb_v$, we have
\begin{align*}
i_{P[\phi,\tilde{X},q,m,1]}d\text{vol}&=-(P[\phi,\tilde{X},q,m,1])^{v} 2(1-\mu)r^2 dud\om\\
  &\gtrsim r^2 \left(C(|\pa_u\phi|^2+|\nabb\phi|^2+ |\phi|^{p+1})-r^{-2}|\phi |^2\right) dud\om,
\end{align*}
Again the lower order term $r^{-2}\phi^2$ can be absorbed by using $ i_{C'V}d\text{vol}$. 
For the integrated local energy estimate \eqref{eq:ILE:doubnull}, consider the region $\cD_{u_0, u_1}^{v_0}$ bounded by $\Si_{u_0}$, $\Si_{u_1}$, $\{r=r_0\}$ and $\Hb_{v_0}$. In view of the energy identity \eqref{eq:energy:id}, the above computations imply that
\begin{align*}
\iint_{\cD_{u_0, u_1}^{v_0}} \mathcal{W}[\phi, p] d\vol + E[\phi, p](\Si_{u_1}^{v_0})+E[\phi, p](\Hb_{v_0}^{u_0, u_1}) \les E[\phi, p](\Si_{u_0}) .
\end{align*}
Estimate \eqref{eq:ILE:doubnull} then follows by letting $v_0\rightarrow \infty$.
\end{proof}

\subsection{The $r$-weighted energy estimate}
In this section, we investigate the $r$-weighted energy estimate for the solution $\phi$. We first establish a lemma showing the components for the current $P[\psi, X, q, m, k]$ for particular choice of $X$, $q$, $m$ and $k$.
\begin{Lem}
\label{rpw}
  Let $f(u)$ be a function of $u$ and $\eta(r)$ be a function of $r$. For constant  $0<\gamma<2$, $c$ and vector fields $X$, $m^{\#}$ and function $q$
  \begin{equation*}
  \begin{split}
  &X=f(u)r^{\gamma}L,\quad q=f(u)(1-\mu)r^{\gamma-1},\quad
  m^{\#}=\Big(\f12  f(u)(1-\mu)\gamma^2 r^{\gamma-2}-f'(u)r^{\gamma-1}\Big)L,
  \end{split}
\end{equation*}
we have
\begin{align*}
  Q[\psi,X,q,m]&=\frac{1}{2}f(u)\gamma r^{\gamma-3}|r\partial_v \psi+\f12 \ga (1-\mu)\psi|^2+\f12 f(u)(2-  \ga +O(\mu))r^{\gamma-1}|\slashed\nabla\psi|^2
    \\&\quad
    -\frac{1}{2}f'(u)(1-\mu)^{-1}r^{\gamma-2}|\partial_{v}(r\psi)|^2
+\frac{1}{2}r^{\ga-3}\big( \ga (1-\frac{\ga}{2})^2+O(\mu) \big)\psi^2,\\
  r^2P_u[\psi, \eta X, \eta q, \eta m, k]&=\eta(r)(1-\mu)f(u)r^{\gamma}\Big(|\slashed\nabla(r\psi)|^2
  +\frac{2kr^2}{p+1}|\psi|^{p+1}
  +((1-\mu)(1-\f12 \ga )\gamma+\mu)\psi^2\Big)\\&
  +\frac{1}{2}\underline{L}\Big(\eta(r)f(u)(1-\mu)r^{\gamma+1}\psi^2\Big)
  +\eta'(r)f(u)(1-\mu)^2r^{\gamma+1}\psi^2,
  \\
  r^2P_v[\psi, \eta X, \eta q, \eta m, k]&=\eta(r)f(u)r^{\gamma}|L(r\psi)|^2
  -\frac{1}{2}L\Big(\eta(r)f(u)(1-\mu)r^{\gamma+1}\psi^2\Big),
\end{align*}
 Here $m^{\#}$ is the dual vector field of the 1-form $m$ and $O(\mu)$ stands for a term that $|O(\mu)|\les \mu$.
\end{Lem}
\begin{proof}
There are pure computations. Let's first recall a formula from \cite{dr3}. For spherically symmetric vector field $V$, we have
\begin{align*}
T_{\alpha\beta}[\psi, 0]\pi_{V}^{\alpha\beta}&=
 (\partial_{u}\psi)^2\partial^u V^{u}
+(\partial_{v}\psi)^2\partial^v V^{v}
-\f12 |\slashed\nabla\psi|^2(\partial^v V_v+\partial^u V_u)
-\frac{V_u-V_v}{2r}(|\slashed\nabla\psi|^2
-\partial^{\alpha}\psi\partial_{\alpha}\psi).
\end{align*}
Apply the above formula to $V=X=f(u)r^{\ga}L$. We have
\[
X^{u}=X_{v}=0, \quad X^{v}=f(u)r^{\ga}, \quad X_u=-2(1-\mu)f(u)r^{\ga}.
\]
As $\partial_v r=-\partial_u r=1-\mu$, we then can compute that
\begin{align*}
T_{\alpha\beta}[\psi,0]\pi_X^{\alpha\beta}&= (\partial_{v}\psi)^2\partial^v (f(u)r^{\ga})
+ |\slashed\nabla\psi|^2 \partial^u ((1-\mu)f(u)r^{\ga})
+(1-\mu)f(u)r^{\ga-1} (|\slashed\nabla\psi|^2
-\partial^{\alpha}\psi\partial_{\alpha}\psi)
\\&
=\frac{1}{2}f(u)\gamma r^{\gamma-1}(\partial_v\psi)^2+f(u)\Big(1-\frac{3}{2}\mu
    -\frac{1}{2}(1-\mu)\gamma\Big)r^{\gamma-1}|\slashed\nabla\psi|^2
    \\&\quad -f(u)(1-\mu)r^{\gamma-1}\partial_{\alpha}\psi\partial^{\alpha}\psi
    -\frac{1}{2}f'(u)(1-\mu)^{-1}r^{\gamma}|\partial_{v}\psi|^2.
\end{align*}
For the lower order terms in $Q[\psi, X, q, m]$, we also need to compute $\nabla^\a m_\a$ and $\Box_g q$. For the divergence of $m$, we compute that
\begin{equation*}
\begin{split}
\nabla^{\alpha}m_{\alpha}&=\frac{1}{(1-\mu)r^2}
\partial_v\big({(1-\mu)r^2m^{v}}\big)
=\partial_v m^v
+\big(\frac{\mu}{r}+\frac{2}{r}(1-\mu)\big)m^v
\\&
=\f12 \ga^2 f(u)  r^{\gamma-3}(1-\mu)\big(\ga-\ga\mu+2\mu\big)+f'(u)r^{\gamma-2}(\mu\gamma-1-\gamma).
\end{split}
\end{equation*}
Next we compute  $\Box_g q$. Note that the function $q$ is spherically symmetric. We in particular have that
\begin{align*}
\Box_g q&=-(1-\mu)^{-1}\partial_{v}\partial_{u}q+\frac{1}{r}(\partial_v-\partial_u)q\\
&=-f'(u)r^{\gamma-2}(\mu(1-\gamma)+\ga)
+f(u)r^{\gamma-3}\Big((\gamma-1)\gamma+\mu(\gamma-1)(3-2\ga)+\mu^2(\gamma-2)^2\Big).
\end{align*}
Since the $r$-weighted energy estimates are of importance for large $r$, the above computations  lead to
\begin{align*}
   Q[\psi, X,q,m]&=T_{\alpha\beta}[\psi,0]\pi_X^{\alpha\beta}
+q\partial^{\alpha}\psi\partial_{\alpha}\psi+\psi m_{\alpha}\partial^{\alpha}\psi
+\frac{1}{2}(\nabla^{\alpha}m_{\alpha}-\Box_g q)\psi^2\\
&=\frac{1}{2}f(u)\gamma r^{\gamma-1}|\partial_v \psi|^2+\f12 f(u)(2-  \ga +O(\mu))r^{\gamma-1}|\slashed\nabla\psi|^2+\psi (m^{v}+ f' r^{\ga-1})\partial_{v}\psi
    \\&\quad
    -\frac{1}{2}f'(u)(1-\mu)^{-1}r^{\gamma-2}|\partial_{v}(r\psi)|^2
+\frac{1}{2}(\nabla^{\alpha}m_{\alpha}-\Box_g q+f'(1-\mu)r^{\ga-2})\psi^2\\
&=\frac{1}{2}f(u)\gamma r^{\gamma-3}|r\partial_v \psi+\f12 \ga (1-\mu)\psi|^2+\f12 f(u)(2-  \ga +O(\mu))r^{\gamma-1}|\slashed\nabla\psi|^2
    \\&\quad
    -\frac{1}{2}f'(u)(1-\mu)^{-1}r^{\gamma-2}|\partial_{v}(r\psi)|^2
+\frac{1}{2}r^{\ga-3}\big( \ga (1-\frac{\ga}{2})^2+O(\mu) \big)\psi^2.
\end{align*}
This gives the expression for $Q[\psi, X, q, m]$.

Next we consider $P[\psi, \eta X, \eta q, \eta m, k]$ for some function $\eta (r)$ of $r$.  Recall that
$$
  P_{\alpha}[\psi,\eta X,\eta q,\eta m,k]
  =T_{\alpha\beta}[\psi,k ]\eta X^{\beta}+\eta q\psi\partial_{\alpha}\psi
  -\frac{1}{2}\partial_{\alpha}(\eta q)\psi^2+\frac{1}{2}\eta m_{\alpha}\psi^2.
$$
For the $v$ component, we derive that
\begin{align*}
  &r^2 P_v[\psi, \eta X, \eta q, \eta m, k]\\
  &=r^2\big(\eta \pa_v \psi X \psi+\eta q\psi\partial_{v}\psi
  -\frac{1}{2}\partial_{v}(\eta q)\psi^2 \big)
  \\&
  =f \eta r^{\gamma+2}(\partial_v\psi)^2
  +f (1-\mu)\eta r^{\gamma+1}\psi\partial_v\psi
  -\frac{1}{2}f r^2 \partial_v\big((1-\mu)r^{\gamma-1}\eta \big)\psi^2\\
  &=f\eta r^{\ga }|\pa_v (r\psi)|^2-\f12 f\left(r^2 \partial_v\big((1-\mu)r^{\gamma-1}\eta \big)\psi^2+(1-\mu)\eta r^{\ga+1}\pa_v (|\psi |^2) +\pa_v (r^2) (1-\mu)r^{\ga-1}\eta |\psi|^2 \right)\\
  &=f\eta r^{\ga }|\pa_v (r\psi)|^2-\f12 f\pa_v\big(r^2  (1-\mu)r^{\gamma-1}\eta  \psi^2\big).
\end{align*}
Here keep in mind that $f$ is a function of $u$ and $X^u=0$.

Similarly for the $u$ component, we can compute that
\begin{equation*}
\begin{split}
  &r^2P_u[\psi, \eta X, \eta q, \eta m, k]\\
  &=r^2\left(\eta \pa_u \psi X \psi +\eta (1-\mu)X^v(\pa^\a\psi \pa_\a\psi +\frac{2k}{p+1}|\psi |^{p+1}) +\eta q\psi\partial_{u}\psi
  -\frac{1}{2}\partial_{u}(\eta q)\psi^2+\frac{1}{2}\eta m_{u}\psi^2
   \right)
  \\&
  =\eta f (1-\mu)r^{\gamma+2}\big(|\slashed\nabla\psi|^2
  +\frac{2k |\psi|^{p+1}}{p+1}\big)+\f12 \pa_u\big(r^2 \eta q \psi^2\big)
  - \f12 |\psi|^2 \Big(  r^2\partial_{u}(\eta q)- r^2\eta m_{u}  + \pa_u(r^2\eta q)\Big) .
\end{split}
\end{equation*}
For the coefficient of $|\psi|^2$, first we have
\begin{align*}
&-\f12   \Big(  r^2\partial_{u}(\eta q)- r^2\eta m_{u}  + \pa_u(r^2\eta q)\Big)=\eta' (1-\mu) r^2 q+\eta r\big((1-\mu)q-r(1-\mu)m^v-r \pa_u q\big).
\end{align*}
Then by the definition of $m$ and $q$, we show that
\begin{align*}
(1-\mu)q-r(1-\mu)m^v-r \pa_u q&=(1-\mu )q-r(1-\mu)(\f12 f \ga^2 (1-\mu)r^{\ga-2}-f'r^{\ga-1})-r \pa_u (f(1-\mu)r^{\ga-1})\\
&=(1-\mu )q-\f12 f(1-\mu)^2  \ga^2 r^{\ga-1} +r(1-\mu) f ((1-\mu)(\ga-1)+\mu )r^{\ga-2}\\
&=f(1-\mu)r^{\ga-1}\big(\mu+ (1-\mu)\ga(1-\frac{\ga}{2})\big).
\end{align*}
Combining the above computations, we then get the expression for $r^2P_u[\psi, \eta X, \eta q, \eta m, k]$.

\end{proof}

We are now ready to establish  the following $r$-weighted energy estimate for the solution $\phi$.
\begin{Prop}
\label{divsk1}
 For all $0<\gamma<\min(2,p-1)$, we have
  \begin{equation*}
  \begin{split}
  &\iint_{\mathcal{D}_{u_1, u_2}\cap\{r\geq R\}} r^{\gamma-3}(|L(r\phi)|^2
  +|\phi|^2+|r\slashed\nabla\phi|^2)
  +r^{\gamma-1}|\phi|^{p+1}d\vol
  +\int_{\mathcal{H}_{u_2}}r^{\gamma}|L(r\phi)|^2dvd\omega\\
  &\lesssim\int_{\mathcal{H}_{u_1}}
  r^{\gamma}|L(r\phi)|^2dvd\omega+E[\phi, p](\Sigma_{u_1}),
  \end{split}
  \end{equation*}
for all $u_2\geq u_1\geq -\frac{R_*}{2}$. Here recall that $R^*=r^*(R)$ and $R$ is a large constant depending on $\ga$.
\end{Prop}
\begin{proof}
Let $\eta(r)$ be an increasing cutoff function of $r$ such that $ \eta=0$ when $r\leq 4M$ and $ \eta=1$ when $r\geq 5M$. Let $f=1$ in the above Lemma \ref{rpw}. For sufficiently large $r$ such that
\begin{align*}
2-\ga+O(\mu)>0,\quad \ga (1-\frac{\ga}{2})^2+O(\mu)>0,
\end{align*}
 the above lemma  then implies that
  \begin{align*}
    Q[\phi,\eta X, \eta q, \eta m]=Q[\phi,  X,  q,  m]&\gtrsim r^{\gamma-3}\big(|r\pa_v\phi+\frac{\ga}{2} (1-\mu)\phi|^2+|r\slashed\nabla\phi|^2+|\phi|^2\big)\\
    & \gtrsim r^{\gamma-3}\big(|\pa_v(r\phi)|^2+|r\slashed\nabla\phi|^2+|\phi|^2\big).
\end{align*}
Since the vector field $X$ is spherically symmetric, we can compute that
\begin{align*}
q-\frac{\div(X)}{p+1}=(1-\mu)r^{\ga-1}-\frac{\pa_v ((1-\mu)r^{\ga+2})}{(p+1)(1-\mu)r^2}=\frac{(p-1-\ga)(1-\mu)-\mu}{(p+1)}r^{\ga-1}
\end{align*}
Therefore for $0<\ga<\min\{p-1, 2\}$ and  sufficiently large $r$, we can bound that
\begin{align*}
  \nabla^{\alpha} P_{\alpha}[\phi,\eta X, \eta q, \eta m, 1]&=\Box_g\phi (X\phi+q\phi)-\frac{1}{p+1} \div (X  |\phi|^{p+1})   +Q[\phi,X,q,m]\\
  &=\left(q-\frac{\div(X)}{p+1}\right)|\phi|^{p+1}+Q[\phi,X,q,m]\\
  &\gtrsim r^{\gamma-3}\big(|\pa_v(r\phi)|^2+|r\slashed\nabla\phi|^2+|\phi|^2+r^2|\phi|^{p+1}\big).
\end{align*}
 Now apply the energy identity \eqref{eq:energy:id} to the vector field $\eta X$, the function  $\eta q$ and 1-form $\eta m$ on the domain $\cD_{u_1, u_2}^{v_0}$ bounded by $\Si_{u_1}$, $\Si_{u_2}$ and $\Hb_{v_0}^{u_1, u_2}$. We obtain the energy identity
\begin{equation*}
\begin{split}
  &\iint_{\mathcal{D}_{u_1, u_2}^{v_0}}\div P dvol+\int_{\underline{\mathcal{H}}_{v_0}^{u_1,u_2}} P_{u}r^2dud\omega+\int_{\mathcal{H}_{u_2}^{v_0}}P_{v}r^2dvd\omega\\
  &=-\int_{\Sigma_{u_1}\cap\{r<R \}}P^{t}r^2drd\omega
  +\int_{\mathcal{H}_{u_1}^{ v_0}}P_{v}r^2dvd\omega
+\int_{\Sigma_{u_2}\cap\{ r<R \}}P^{t}r^2drd\omega.
\end{split}
\end{equation*}
Here $P$ is short for $P[\phi, \eta X, \eta q, \eta m, 1]$.
In view of the above lemma \ref{rpw}, on the out going null hypersurface $\H_{u}$, the term $P_{v}r^2$ is nonnegative except the term $-\frac{1}{2} L(\eta f (1-\mu)r^{\ga+1}|\phi|^2)$ which could be canceled  through integration by parts. Indeed, since the region $\cD_{u_1, u_2}^{v_0}$ is closed and $\eta=0$ when $r\leq 4M$, the boundary of $\cD_{u_1, u_2}^{v_0}$ is piece wise smooth. On the incoming null hypersurface $\Hb_{v_0}$, there is a similar term $\frac{1}{2}\Lb(\eta f (1-\mu)r^{\ga+1}|\phi|^2)$
in the expression for $r^2P_u$. Now note that
\begin{align*}
r^2 P^t=r^{2} (P^u+P^v)=-\frac{1}{2}(1-\mu)^{-1}{ r^2}(P_u+P_v).
\end{align*}
Hence the associate term in the expression for $r^2P^t$ on $\Si_{u}\cap \{r\leq 5M\}$ is
\begin{align*}
-\frac{1}{2(1-\mu)} \frac{\Lb-L}{2}(\eta f (1-\mu)r^{\ga+1}|\phi|^2)=\frac{1}{2(1-\mu)} \pa_{r^*}(\eta f (1-\mu)r^{\ga+1}|\phi|^2)=\frac{1}{2}\pa_{r}(\eta f (1-\mu)r^{\ga+1}|\phi|^2).
\end{align*}
Therefore on the boundary of $\cD_{u_1, u_2}^{v_0}$, such terms could be canceled.

As demonstrated that for sufficiently large $r$ ($r\geq R$ and $R$ sufficiently large), $\div(P)$ has a nonnegative lower bound. For finite $r$, these terms can be bounded by using the integrated local energy estimate of Proposition \ref{mainenergy}. More precisely since $\eta$ is smooth and vanishes when $r\leq 4M$, we derive from the integrated local energy estimate  \eqref{eq:ILE:doubnull} that
\begin{align*}
&\iint_{\cD_{u_1, u_2}^{v_0}\cap \{r\leq R\}}|\div (P)|d\vol +\int_{\Si_{u_2}\cap \{r\leq R\}}r^2 |P^t| drd\om\\
&\les \int_{\cD_{u_1, u_2}}\mathcal{W}[\phi, p]d\vol +E[\phi, p](\Si_{u_2})\les E[\phi, p](\Si_{u_1}).
\end{align*}
The proposition then follows by letting $v_0\rightarrow \infty$.

\end{proof}

\subsection{A spacetime bound for $\phi$ and $Z\phi$}
Based on the above $r$-weighted energy estimate and the integrated local energy estimates, we derive a spacetime bound for the solution $\phi$ as well as its derivatives $Z\phi$.
\begin{Prop}
\label{H2}
     For $\frac{1+\sqrt{17}}{2}<p<5$ and any admissible pair $(p_1, q_1)$,  we have
\begin{equation*}
\begin{split}
 E[Z^k\phi](\Sigma_u)+\|Z^k\phi\|_{LE(\cD_{u})}^2
+\|Z^k\phi\|_{L_{\widetilde{v}}^{p_1}L_x^{q_1}(\cD_{u})}&\lesssim   E[Z^{\leq k} \phi](\Sigma_{u_0}),\\
 E[Z^k\phi]({\widetilde{\Sigma}}_{\widetilde{v}})+\|Z^k\phi\|_{LE(\tcD_{\tv})}^2
+\|Z^k\phi\|_{L_{\widetilde{v}}^{p_1}L_x^{q_1}(\tcD_{\tv})}&\lesssim E[Z^{\leq k}\phi]({\widetilde{\Sigma}}_{\widetilde{v}_0})
 \end{split}
\end{equation*}
for all $u\geq u_0$, $\tv\geq \tv_0$ and $k\leq 1$.
\end{Prop}
Note that $\Box_g$ commutes with the vector fields $Z\in \Ga$. The above spacetime bounds rely on the following spacetime estimate for the nonlinearity.

\begin{Lem}
\label{k22}
If $\frac{1+\sqrt{17}}{2}<p<5$, there exists pair $(k, l)$  such that
  \begin{equation*}
  1<k<2, \quad \frac{1}{k}+\frac{3}{l}=2, \quad \||\phi|^{p-1}\|_{L_{\widetilde{v}}^{k}L_x^{l}(\mathcal{M})}\lesssim 1.
\end{equation*}
Here by our notation, the implicit constant also relies on the initial energy $E[\phi, p](\tSi_0)$.
\end{Lem}
\begin{proof}
Let's first show that there exists some  $k_0\in (1, 2)$ such that $\phi \in L_{\tv}^{k_0(p-1)} L_x^{k_0(p-1)}(\mathcal{M})$. For the case when $3<p<5$, we simply rely on the $r$-weighted energy estimate of Proposition \ref{divsk1} with $\ga=1$ together with the integrated local energy estimate of Proposition \ref{mainenergy}. We can show that
\begin{align*}
\iint_{\mathcal{M}}|\phi|^{p+1}d\vol =\iint_{\mathcal{M}\cap \{r\leq R\}}|\phi|^{p+1}d\vol +\iint_{\mathcal{M}\cap \{r\geq R\}}|\phi|^{p+1}d\vol \les E[\phi, p](\tSi_{0}).
\end{align*}
In particular for this case we can take $k_0=\frac{p+1}{p-1}$.

For the case when $\frac{1+\sqrt{17}}{2}<p\leq 3$, let $a$, $k_0$, $\ga_0$ be constants such that
\begin{align*}
&\frac{a}{p+1}+\frac{k_0(p-1)-a}{2}=1,\quad 1<k_0<2,\quad 0<\ga_0<p-1,\quad  0<a<p+1,\\
& \frac{(\ga_0-1)(k_0(p-1)-2)}{p+1-k_0(p-1)}-(3-\ga_0)\geq 0.
\end{align*}
Note that left hand side of  the second inequality is increasing in terms of $k_0$ and $\ga_0$. The existence of such $k_0$ and $\ga_0$ follows by the fact that the left hand side of the second inequality is positive when $k_0=2$ and $\ga_0=p-1$. This is guaranteed by the condition condition that $p>\frac{\sqrt{17}+1}{2}$. For such $\ga_0$, the $r$-weighted energy estimate of Proposition \ref{divsk1} implies that
\begin{equation*}
\int_{\mathcal{M}} r^{\gamma_0-3}|\phi|^2+r^{\gamma_0-1}|\phi|^{p+1}d\vol
\lesssim E[\phi, p](\widetilde{\Sigma}_0).
\end{equation*}
We thus can bound that
\begin{align*}
\iint_{\mathcal{M}} |\phi|^{k_0(p-1)}d\vol &= \| |\phi|^{ a}r^{\frac{a(\ga_0-1)}{p+1}}|\phi|^{k_0(p-1)-a}
r^{-\frac{a(\ga_0-1)}{p+1}}\|_{L^1}\\
&\leq \| |\phi|^{ a}r^{\frac{a(\ga_0-1)}{p+1}} \|_{L^{\frac{p+1}{a}}} \| |\phi|^{k_0(p-1)-a}
r^{-\frac{a(\ga_0-1)}{p+1}}\|_{L^{\frac{2}{k_0(p-1)-a}}}\\
& = \| |\phi|^{ p+1} r^{\ga_0-1} \|_{L^{1}}^{\frac{a}{p+1}} \| |\phi|^{2}
r^{-\frac{(\ga_0-1)(k_0(p-1)-2)}{p+1-k_0(p-1)} }\|_{L^{1}}^{\frac{k_0(p-1)-a}{2}}\\
&\les E[\phi, p](\tSi_{0}).
\end{align*}
Hence in any case, we have shown that  $\phi \in L_{\tv}^{k_0(p-1)} L_x^{k_0(p-1)}(\mathcal{M})$ for some $k_0\in (1, 2)$.

Now for such constant $k_0<2$, let $\theta $, $p_1$, $ q_1$ be constants such that
\begin{align*}
 (\frac{4}{k_0}-\frac{p-1}{2} )\theta=\frac{5-p}{2}, \quad  \frac{\theta}{k_0}+\frac{(p-1)(1-\theta)+1}{p_1}=1, \quad
  \frac{\theta}{k_0}+\frac{(p-1)(1-\theta)+1}{q_1}=\frac{1}{2}.
\end{align*}
One can check that $0<\theta<1$ and $(p_1, q_1)$ is admissible, that is, $\frac{1}{p_1}+\frac{3}{q_1}=\frac{1}{2}$.
By using H\"{o}lder's inequality, we can bound that
\begin{equation*}
  \|\Box_g \phi\|_{L_{\widetilde{v}}^{1}L_x^{2}(\tcD_{\tv_1, \tv_2})}= \||\phi|^{p-1}\phi \|_{L_{\widetilde{v}}^{1}L_x^{2}(\tcD_{\tv_1, \tv_2})}\leq
  \|\phi\|_{L_{\widetilde{v}}^{k_0(p-1)}L_x^{k_0(p-1)}(\tcD_{\tv_1, \tv_2})}^{(p-1)\theta}
  \|\phi\|_{L_{\widetilde{v}}^{p_1}L_x^{q_1}(\tcD_{\tv_1, \tv_2})}^{(p-1)(1-\theta)+1},
\end{equation*}
Hence in view of the Strichartz estimate of Proposition \ref{prop:IS:v}, for any admissible pair $(\tilde{p}_1, \tilde{q}_1)$, we derive that
\begin{align*}
    \|\phi\|^2_{L_{\tv}^{\tilde{p}_1}L_{x}^{\tilde{q}_1}(\tcD_{\tv_1, \tv_2})}
    \les E[\phi](\widetilde{\Sigma}_{\widetilde{v}_1})+  \|\phi\|_{L_{\widetilde{v}}^{k_0(p-1)}L_x^{k_0(p-1)}(\tcD_{\tv_1, \tv_2})}^{ 2(p-1)\theta}
  \|\phi\|_{L_{\widetilde{v}}^{p_1}L_x^{q_1}(\tcD_{\tv_1, \tv_2})}^{ 2((p-1)(1-\theta)+1)}.
\end{align*}
Since the energy $E[\phi](\widetilde{\Sigma}_{\widetilde{v}_1}) $ is uniformly bounded by the initial energy $E[\phi, p](\tSi_0)$, in view of Lemma \ref{main1} and by letting $(\tilde{p}_1, \tilde{q}_1)=(p_1, q_1)$, we in particular conclude that
\begin{align*}
\|\phi\|_{L_{\widetilde{v}}^{p_1}L_x^{q_1}(\mathcal{M})}\les 1,
\end{align*}
which in turn shows that
\begin{align*}
\|\phi\|_{L_{\widetilde{v}}^{\tilde{p}_1}L_x^{\tilde{q}_1}(\mathcal{M})}\les 1
\end{align*}
for all admissible pair $(\tilde{p_1}, \tilde{q}_1)$.

With these preparations, we are now ready to prove the Lemma. For the case when
$\frac{1+\sqrt{17}}{2}<p\leq 3$, the Lemma follows by interpolation. Indeed, since $|\phi|^{p-1}\in L^{k_0}L^{k_0}$ for some $k_0<2$ (by our construction, for this case $k_0$ can be sufficiently close to $2$)  and  $|\phi|^{p-1}\in L^{\frac{\tilde{p}_1}{p-1}}L^{\frac{\tilde{q}_1}{p-1}}$ for all admissible pair $(\tilde{p}_1, \tilde{q}_1)$, we conclude that $|\phi|^{p-1}\in L^k L^l$ for all $0\leq \theta\leq 1$ with
  \begin{equation*}
    \frac{1}{k}=\frac{\theta}{k_0}+\frac{(1-\theta)(p-1)}{\tilde{p}_1},
    \quad
    \frac{1}{l}=\frac{\theta}{k_0}+\frac{(1-\theta)(p-1)}{\tilde{q}_1}.
  \end{equation*}
  Let
  \begin{align*}
  \tilde{p}_1=2(p-1),\quad \tilde{q}_1=\frac{6(p-1)}{p-2},\quad \theta =\frac{5-p}{\frac{8}{k_0}-p+1}.
    \end{align*}
    Then we have
  \begin{align*}
  \frac{1}{k}=\frac{1}{2}+(\frac{1}{k_0}-\frac{1}{2})\theta, \quad \frac{1}{k}+\frac{3}{l}=2.
  \end{align*}
  Since $k_0<2$ and $p\leq 3$, we see that $0<\theta <1$. Moreover as we can choose $k_0<2$ sufficiently close to $2$, we thus conclude that $\f12 <\frac{1}{k}<1$, that is $1<k<2$. Hence the Lemma holds for the case when $p\leq 3$.

 For the case when $3<p<5$, the classical energy estimate and the $r$-weighted energy estimate with $\ga=1$ show that
\begin{align*}
\|\phi\|_{L_{\tv}^\infty L_x^{p+1}}\les 1,\quad \|\phi\|_{L_{\tv}^{p+1} L_x^{p+1}}\les 1.
\end{align*}
In particular $|\phi|^{p-1}\in L_{\tv}^{\frac{s}{p-1}}L_x^{\frac{p+1}{p-1}}$ for all $s\geq p+1$. Similarly as $|\phi|^{p-1}\in L^{\frac{\tilde{p_1}}{p-1}}L^{\frac{\tilde{q_1}}{p-1}}$ for all admissible pair $(\tilde{p}_1,\tilde{q}_1)$, we see that $|\phi|^{p-1}\in L^k L^l$ with
  \begin{equation*}
    \frac{1}{k}=\frac{\theta(p-1)}{s}+\frac{(1-\theta)(p-1)}{\tilde{p}_1},
    \quad
    \frac{1}{l}=\frac{\theta(p-1)}{p+1}+\frac{(1-\theta)(p-1)}{\tilde{q}_1}
  \end{equation*}
The condition $\frac{1}{k}+\frac{3}{l}=2$ is then reduced to
 $$\frac{5-p}{2} =\theta\Big(\frac{p-1}{s}
  +\frac{3(p-1)}{p+1}-\frac{p-1}{2} \Big).$$
  Since $p>3$, we can choose  $s\in [p+1, 2(p-1))$, which ensures that
  $\theta\in(0,1)$. It then remains to show that $1<k<2$ for some $s$ and admissible pair $(\tilde{p_1}, \tilde{q}_1)$. This can be realized by letting  $\tilde{p}_1=2(p-1)$ and $s$ sufficiently close to $2(p-1)$. We hence finished the proof for the lemma.
\end{proof}

We now can prove Proposition \ref{H2}. Note that for $k\leq 1$
\begin{align*}
|\Box_g Z^k \phi| =|Z^k(|\phi|^{p-1}\phi)|\les |\phi|^{p-1}|Z^k\phi|.
\end{align*}
As for the foliation $\tSi$, in view of Proposition \ref{prop:IS:v}, for any admissible pair $(p_1, q_1)$, it holds that
\begin{equation*}
    \sup_{\widetilde{v}\in [\widetilde{v}_1,\widetilde{v}_2]}E[Z^k\phi](\widetilde{\Sigma}_{\widetilde{v}} )
    +\|Z^k\phi\|_{LE(\tcD_{\tv_1, \tv_2})}+\|Z^k\phi\|^2_{L_{\tv}^{p_1}L_{x}^{q_1}(\tcD_{\tv_1, \tv_2})} \lesssim E[Z^k \phi ](\widetilde{\Sigma}_{\widetilde{v}_1})
    +\|\Box_g Z^k \phi\|^2_{L_{\tv}^1L^2_x(\tcD_{\tv_1, \tv_2})}
\end{equation*}
for all $\tv_2\geq \tv_1$. The above Lemma shows that $|\phi|^{p-1}\in L^k_{\tv} L^l_{x}$ for some $(k, l)$ verifying $\frac{1}{k}+\frac{3}{l}=2$. The second estimate of Proposition \ref{H2} adapted to the foliation $\tSi$ then follows from Lemma \ref{main}. The first estimate associated to the foliation $\Si$ follows in a similar way in view of Proposition \ref{L1L2foru}.

As  $Z\in \{\partial_{\widetilde{v}}, \Omega_{ij}\}$, a standard trace theorem (see \eqref{eq:trace}) and Sobolev embedding on the unit sphere give  a rough pointwise decay estimate for the solution
  \begin{equation*}
|\phi|\les   r^{-\frac{1}{2}}\sqrt{\cE_1}.
\end{equation*}

\section{Energy flux decay}
\label{rw}
The previous section only gives the spatial decay of the solution as the commuting vector fields $Z$ do not contain weights in time. To obtain time decay of the solution, we rely on the method of Dafermos-Rodnianski in \cite{newapp}. In this section, we show the energy flux decay for $E[\phi](\Si_u)$.
\begin{Prop}
\label{ephidecay0}
  If  $\frac{1+\sqrt{17}}{2}<p<5$, then for all $1<\ga_0<\min\{2, p-1\}$ and $N> 2$ we have
  \begin{align*}
  E[\phi, p](\Sigma_u)\lesssim
  u_+^{-\gamma_0}  +u_+^{- N/2-1} E[Z^{\leq 1}\phi](\Si_u).
\end{align*}
Here $u_+=\sqrt{ 4+u^2}$ and the implicit constant also relies on $N$.
\end{Prop}
To prove this proposition we need the following lemma.
\begin{Lem}
\label{lp}
 Let  $f(x)$ be a function supported  in $ \{\frac{1}{2}\leq |x|\leq \frac{3}{2}, x\in {\mathbb{R}}^{3}\}$. For constants $q>2$, $N>1$ and $\forall\  t\geq2$, we have
  \begin{equation*}
  \|f\|_{L^2(R^3)}
  \leq  C\Big(\ln t\||\ln { ||x|-1|}|^{-1}f\|_{L^2(R^3)}+t^{- N\frac{q-2}{2q}}\|f\|_{L^q(R^3)}\Big)
  \end{equation*}
  for some constant $C$ depending only on $q$ and $N$.
\end{Lem}
\begin{proof}
  We divide the region $ \{\frac{1}{2}\leq |x|\leq \frac{3}{2}\}$ into two parts  $\Omega_1=\{ ||x|-1|\leq t^{-N}\}$ and $\Omega_2=\{t^{-N}\leq  ||x|-1|\leq \frac{1}{2}\}$.
Then
  \begin{align*}
    \|f\|_{L^2(R^3)}&=\|f\|_{L^2(\Omega_1)}+\|f\|_{L^2(\Omega_2)}
    \\&\leq \|1\|_{L^{\frac{2q}{q-2}}(\Omega_1)}\|f\|_{L^q(\Omega_1)}+N\|\ln t|\ln { ||x|-1|}|^{-1}f\|_{L^2(\Omega_2)}
    \\&
    \leq C(t^{- N\frac{q-2}{2q}}\|f\|_{L^q(R^3)}+\ln t\||\ln { ||x|-1|}|^{-1}f\|_{L^2(R^3)}).
  \end{align*}
\end{proof}
We now can prove the above Proposition.
\begin{proof}[\textbf{Proof of Proposition \ref{ephidecay0}}]
In the $r$-weighted energy estimate in Proposition \ref{divsk1}, let $\gamma=\gamma_0$ and $u_1=-\f12 R^{*}$. Since for simplicity we have assumed that the data are supported in $\{|x|\leq R\}$ on $\tSi_0$, the finite speed of propagation implies that the solution vanishes on $\H_{-\f12 R^*}$. We therefore derive that
\begin{equation*}
\begin{split}
\int_{-\f12 R^{*}}^{\infty}&
\int_{\mathcal{H}_u}
r^{\gamma_0-1} |L(r\phi)|^2 dvd\omega du\les E[\phi, p](\widetilde{\Sigma}_0),\quad  \int_{\mathcal{H}_{u}}r^{\gamma_0}|L(r\phi)|^2 dvd\omega
 \lesssim E[\phi, p](\widetilde{\Sigma}_0)
\end{split}
\end{equation*}
for all $u\geq -\f12 R^{*}$, from which we can extract a dyadic sequence $\{u_k\}_{3}^{\infty}$ such that
\begin{equation*}
  u_3=R^{*},\quad 2u_k\leq u_{k+1}\leq \Lambda u_k,\quad
  \int_{\mathcal{H}_{u_k}}r^{\gamma_0-1}|L(r\phi)|^2dvd\omega\lesssim u_k^{-1} E[\phi,p](\widetilde{\Sigma}_0)
\end{equation*}
for some constants $\Lambda$ depending only on $p$ and $\gamma_0$.  Since $1<\ga_0<2$,  interpolation then leads to
\begin{equation}
\label{eq:ukdecay}
\int_{\mathcal{H}_{u_k}}r|L(r\phi)|^2dvd\omega\lesssim u_k^{1-\gamma_0} E[\phi,p](\widetilde{\Sigma}_0),\ k\geq 3.
\end{equation}
This controls the right hand side of the $r$-weighted energy estimate of Proposition \ref{divsk1}. To recover the classical energy, we need to bound the difference between $r |L\phi|^2$ and $|L(r\phi)|^2$. First we have
\begin{equation*}
  r^2|L\phi|^2=
   |L(r\phi)|^2
  -(1-\mu)^2|\phi|^2-2(1-\mu)r\phi L\phi=|L(r\phi)|^2-L(r(1-\mu)|\phi|^2)+\mu(1-\mu)|\phi|^2.
\end{equation*}
We therefore can bound that
\begin{equation*}
\begin{split}
  \int_{u_1}^{u_2}\int_{\mathcal{H}_u}|L\phi|^2 r^2dudvd\om&
  \lesssim
  \int_{u_1}^{u_2}\int_{\mathcal{H}_u} |L(r\phi)|^2+ \mu \phi^2 dudvd\om +\int_{2u_1+R^*}^{ 2u_2+R^*}\int_{r=R } r(1-\mu)|\phi|^2 d\om dt
  \\& \les \int_{u_1}^{u_2}\int_{\mathcal{H}_u} |L(r\phi)|^2+  \phi^2 dudvd\om
  +\int_{2u_1+R^{*}}^{2u_2+R^{*}}\int_{ r\leq R}|\partial_{r}\phi|^2+\phi^2drd\omega d\widetilde{v}.
\end{split}
\end{equation*}
The second term on the right hand side of the above inequality can be bounded by using the integrated local energy estimate.
Now using the logarithmic integrated local energy estimate of Proposition \ref{H2} and the $r$-weighted energy estimate of Proposition \ref{divsk1} with $\ga=1$, we can bound that
 \begin{equation*}
\begin{split}
  \int_{u_1}^{u_2}E[\phi, p](\Sigma_u)du
  &=\int_{u_1}^{u_2} \int_{r\leq R }|\partial\phi|^2  +|\phi|^{p+1}d\vol +\int_{u_1}^{u_2}\int_{\mathcal{H}_u}|L\phi|^2
  +|\slashed\nabla\phi|^2+{|\phi/r|^2}+ |\phi|^{p+1} d\vol\\
  &\les \int_{u_1}^{u_2} \int_{r\leq R }|\partial\phi|^2  +|\phi|^{p+1}d\vol +\int_{u_1}^{u_2}\int_{\mathcal{H}_u}|L(r\phi)|^2
  +r^2|\slashed\nabla\phi|^2+|\phi|^2+ r^2|\phi|^{p+1}  dudvd\om\\
   &\les \int_{u_1}^{u_2} \int_{|r*|\leq \f12  }|\partial_{\tv}\phi|^2  +|\nabb\phi|^{2}d\vol +\int_{\mathcal{H}_{u_1}}
  r|L(r\phi)|^2dvd\omega+E[\phi, p](\Sigma_{u_1}),
  \end{split}
\end{equation*}
Here without loss of generality, we may assume that $M>1$.
For the integral near the photon sphere, we use the above Lemma \ref{lp} together with the Strichartz estimate of Proposition \ref{H2}. In fact on $\Si_u\cap \{|r^*|\leq \f12 \}$, Lemma \ref{lp} shows that
\begin{align*}
 \int_{ |r^*|\leq \f12}|(\partial_{\tv}, \slashed\nabla)\phi|^2 dx
  \lesssim  \int_{ |r^*|\leq \f12 }|\ln u_{+}|^2|\ln |r^{*}||^{-2}|(\partial_{\tv},\slashed\nabla)\phi|^2 dx
 +u_{+}^{- 5 N/6} \|( \partial_{\tv},\slashed\nabla)\phi\|_{L_x^{12}( |r^*|\leq \f12)}^2.
\end{align*}
Note that $|1-\frac{3M}{r}|\geq  |r^*|$ when $|r^*|\leq \f12$. In particular, we can bound the first term on the right hand side of the above inequality by using the logarithmic integrated local energy estimate of Proposition \ref{H2}. More precisely, we can show that
\begin{align*}
\int_{u_1}^{u_2}\int_{ |r^*|\leq \f12 }|\ln u_{+}|^2|\ln |r^{*}||^{-2}|(\partial_{\tv},\slashed\nabla)\phi|^2 d\vol &\les |\ln (u_2)_+|^2  \int_{u_1}^{u_2}\int_{ |r^*|\leq \f12 } | 1-\ln |1-\frac{3M}{r}||^{-2}|(\partial_{\tv},\slashed\nabla)\phi|^2 d\vol\\
&\les |\ln (u_2)_+|^2 E[\phi ](\Si_{u_1}).
\end{align*}
For the second term we make use of the Strichartz estimate of Proposition \ref{H2} with admissible pair $(4, 12)$
$$
\int_{u_1}^{u_2}u_{+}^{- 5 N/6}\|(\partial_{\tv},\slashed\nabla)\phi\|_{L_x^{12}}^2 du
\lesssim (u_1)_{+}^{- 5 N/6+\f12 }E[Z^{\leq 1}\phi](\Sigma_{u_1})
\lesssim (u_1)_{+}^{- N/2}E[Z^{\leq 1}\phi]( \Sigma_{u_1}).
$$
 Combining these estimates, we have shown that
\begin{align*}
\int_{u_1}^{u_2}E[\phi, p](\Sigma_u)du
\lesssim
  |\ln (u_2)_{+}|^2E[\phi, p](\Sigma_{u_1})
  +\int_{\mathcal{H}_{u_1}}r|L(r\phi)|^2dvd\omega
  +(u_1)_{+}^{- N/2}E[Z^{\leq 1}\phi]( \Sigma_{u_1}).
\end{align*}
Since the energy flux $E[\phi, p](\Sigma_u)$ is decreasing up to a constant in the sense of Proposition \ref{mainenergy} and
\begin{align*}
\int_{\mathcal{H}_{u_1}}r|L(r\phi)|^2dvd\omega\les \int_{\mathcal{H}_{u_1}}r^{\ga_0}|L(r\phi)|^2dvd\omega \les E[\phi, p](\tSi_0),
\end{align*}
 we in particular derive that
\begin{equation*}
\begin{split}
  \frac{1}{2}u&E[\phi, p](\Sigma_{u})
  \lesssim \int_{\frac{1}{2}u}^{u}
  E[\phi, p](\Sigma_{u'})du'
     \lesssim
  (\ln u_{+})^2E[\phi](\widetilde{\Sigma}_{0})+u_+^{- N/2}E[Z^{\leq 1}\phi](\Si_{u}).
  \end{split}
\end{equation*}
for all $u>1$. Here we may constantly use the fact $E[\phi, p](\tSi_0)\les E[\phi](\tSi_0)$ about the initial data.
This leads to a weak decay estimate for the energy flux
\begin{equation*}
  E[\phi, p](\Sigma_u)\lesssim u_+^{-1} (\ln u_+)^2 E[\phi](\widetilde{\Sigma}_{0})+u_+^{- N/2-1}  E[Z^{\leq 1}\phi]( {\Sigma}_{u}),\quad \forall u\geq -\frac{R^*}{2}.
\end{equation*}
 Using this weak decay estimate together with the $r$-weighted energy flux \eqref{eq:ukdecay} decay on the time sequence $u_k$, we then can derive that
 \begin{align*}
  &(u_{k+1}-u_{k})E[\phi, p](\Sigma_{u_{k+1}})\lesssim \int_{u_{k}}^{u_{k+1}}E[\phi,p](\Sigma_{u})du \\
  &\lesssim \ln(u_{k+1})^2E[\phi,p](\Sigma_{u_k})+
  \int_{\mathcal{H}_{u_k}}r|L(r\phi)|^2dvd\omega
  +u_k^{- N/2}E[Z^{\leq 1}\phi](\widetilde{\Sigma}_{0})
  \\&
  \lesssim u_{k}^{1-\gamma_0}E[ \phi](\widetilde{\Sigma}_{0})+u_{k}^{- N/2}E[ Z^{\leq 1}\phi]( {\Sigma}_{u_k}).
\end{align*}
Here note that $1<\ga_0<2$. Since the sequence $u_k$ is dyadic, for general $u\in [u_k, u_{k+1}]$, we derive that
\begin{align*}
E[\phi, p](\Si_u)\les E[\phi, p](\Si_{u_{k}})&\les u_{k-1}^{-\ga_0}E[ \phi](\tSi_0)+u_{k-1}^{- N/2-1}E[Z^{\leq 1}\phi](\Si_{u_{k-1}})\\
&\les u_{+}^{-\ga_0}E[ \phi](\tSi_0)+u_{+}^{- N/2-1}E[Z^{\leq 1}\phi](\Si_u) .
\end{align*}
We thus finished the proof for Proposition \ref{ephidecay0}.
\end{proof}

\section{Higher order energy estimates}
To obtain improved decay estimates for the solution, we also need higher order energy decay estimates. Similar to the energy flux decay of $E[\phi](\Si_u)$ in the previous section, the decay of $E[Z\phi](\Si_u)$ relies on the boundedness of $E[Z^2\phi](\Sigma_u)$ due to the degeneracy of the integrated local energy estimate.
\subsection{Strichartz estimate for $Z^2\phi$}
We begin with the Strichartz estimates for the second order derivatives of the solutions.
Unlike the uniform spacetime bound for the lower order derivatives of the solution obtained in Proposition \ref{H2}, the second derivatives of the solution $Z^2\phi$ may grow in time.
\begin{Lem}
\label{Z2phi}
  For $\frac{1+\sqrt{17}}{2}<p<5$ and admissible pair $(p_1, q_1)$, we have
  \begin{equation*}
  E[Z^2\phi](\tSi_{\tv_1})
  + \|Z^2\phi\|_{L_{\tv}^{p_1}L_x^{q_1}(\tcD_{\tv_0,\tv_1})}^2
  \lesssim (\widetilde{v}_1)_{+}^{ 2} E[Z^{\leq 2}\phi](\tSi_0).
\end{equation*}
\end{Lem}
\begin{proof}
For $Z\in \{\pa_{\tv}, \Om_{ij}\}$, consider the equation
\begin{equation*}
\Box_g(Z^2\phi)=p(p-1)|\phi|^{p-3}\phi Z\phi Z\phi+p|\phi|^{p-1}Z^2\phi.
\end{equation*}
Here $Z$ may vary in different places.
We see that the equation for $Z^2\phi$ is similar to that for $Z\phi$, except that there is a new lower order nonlinear term $|\phi|^{p-3}\phi Z\phi Z\phi$. The proof then goes similar to that for Proposition \ref{H2} once we have controlled this lower order nonlinear term.

First by using the Strichartz estimate of Proposition \ref{H2} for $Z\phi$, we can bound that
\begin{align*}
\||\phi|^{p-2}|Z\phi| |Z\phi|\|_{L_{\tv}^1L_x^2(\tcD_{\tv_0, \tv_1})}&\les \|Z\phi\|_{L_{\tv}^\infty L_x^6(\tcD_{\tv_0, \tv_1})}^2 \||\phi|^{p-2} \|_{L_{\tv}^{1}L_x^6 (\tcD_{\tv_0, \tv_1})}\\
&\les { E[Z^{\leq 1}\phi](\tSi_0)} \|\phi  \|^{p-2}_{L_{\tv}^{p-2}L_x^{6(p-2)} (\tcD_{\tv_0, \tv_1})}.
\end{align*}
Now on the initial hypersurface, we can apply Gargliado-Nirenberg inequality to bound that
\begin{align*}
E[Z^{\leq 1}\phi](\tSi_0)^2\les E[Z^{\leq 2}\phi](\tSi_0) E[\phi](\tSi_0)\les E[Z^{\leq 2}\phi](\tSi_0).
\end{align*}
Here we may need to use the equation \eqref{maineq} to recover the missing direction $\pa_r\notin \{\pa_{\tv}, \Om_{ij}\}$.

In view of the proof for Lemma \ref{k22} and the Strichartz estimate for $\phi$ in Proposition \ref{H2}, we recall that
\begin{align*}
\|\phi\|_{L_{\tv}^{p_2}L_x^{q_2}}+\|\phi\|_{L_{\tv}^{k_0(p-1)}L_x^{k_0(p-1)}}\les 1
\end{align*}
for all admissible pair $(p_2, q_2)$ and some $1<k_0<2$. Since $p>\frac{1+\sqrt{17}}{2}>\frac{5}{2}$, we in particular conclude that for any $q_2> \max(6,6(p-2))$ there exists $\theta\in (0, 1)$ such that
\begin{align*}
\frac{\theta}{q_2}+\frac{1-\theta}{k_0(p-1)}=\frac{1}{6(p-2)}.
\end{align*}
Now let  $n_1$ be the constant  
\begin{equation*}
\frac{\theta}{p_2}+ \frac{1-\theta}{k_0(p-1)}=\frac{1}{n_1}.
\end{equation*}
Then by $ \frac{1}{p_2}+\frac{3}{q_2}=\frac{1}{2}$ we have $\frac{\theta}{2}+ \frac{4(1-\theta)}{k_0(p-1)}=\frac{1}{n_1}+\frac{1}{2(p-2)}$.
If $3\leq p<5$ we have $ \frac{3}{2(p-2)}>\frac{1}{2}$, we can choose $q_2$ close to $6(p-2) $ then $ \theta$ is  close to $1$, $ \frac{\theta}{2}+ \frac{4(1-\theta)}{k_0(p-1)}$ is  close to $\frac{1}{2}$, say, $ <\frac{3}{2(p-2)}$. Then $\frac{1}{n_1}+\frac{1}{2(p-2)}<\frac{3}{2(p-2)} $, $n_1>p-2$.
If $\frac{1+\sqrt{17}}{2}< p\leq3$, we can choose $q_2$ close to $6 $ such that $ p_2>k_0(p-1)$, we can also choose $k_0$ close to $2$ then $n_1>k_0(p-1)>p-2$. Thus we always have $n_1>p-2$.

Then by interpolation and H\"older's inequality, we can bound that
\begin{align*}
\|\phi  \|^{p-2}_{L_{\tv}^{p-2}L_x^{6(p-2)} (\tcD_{\tv_0, \tv_1})}\les (1+|\tv_1|)^{1-\frac{p-2}{n_1}}\|\phi  \|^{p-2}_{L_{\tv}^{n_1}L_x^{6(p-2)} (\tcD_{\tv_0, \tv_1})}\les (\tv_1)_+.
\end{align*}
Now by Strichartz estimate of Proposition \ref{prop:IS:v}, for any admissible pair $(p_1, q_1)$ we have
\begin{align*}
  &\sup_{\widetilde{v}_0\leq\widetilde{v}\leq \widetilde{v}_1}E[Z^2\phi](\widetilde{\Sigma}_{\widetilde{v}})
  +\|Z^2\phi\|^2_{L_{\widetilde{v}}^{p_1}L_x^{q_1}(\tcD_{\tv_0, \tv_1})}\\
  &\lesssim
  E[Z^2\phi](\widetilde{\Sigma}_{\widetilde{v}_0})
  +\||\phi|^{p-1}Z^2\phi\|^2_{L_{\widetilde{v}}^{1}L_x^{2}(\tcD_{\tv_0, \tv_1})}+
  \||\phi|^{p-2} Z\phi Z\phi\|^2_{L_{\widetilde{v}}^{1}L_x^{2}(\tcD_{\tv_0, \tv_1})}\\
  &\lesssim
  (\tv_1)_+^{ 2} E[Z^2\phi](\widetilde{\Sigma}_{\widetilde{v}_0})
  +\||\phi|^{p-1}Z^2\phi\|^2_{L_{\widetilde{v}}^{1}L_x^{2}(\tcD_{\tv_0, \tv_1})}.
\end{align*}
Since $|\phi|^{p-1}\in L_{\tv}^k L_x^l$ for some $1<k<2$, $ \frac{1}{k}+\frac{3}{l}=2$ by Lemma \ref{k22}, the Proposition then follows by using
 Lemma \ref{main}.
\end{proof}

\subsection{Improved energy flux decay}
To improve the pointwise decay estimate for the solution, our strategy is to improve the decay for the energy flux $E[Z^k\phi](\Si_u)$. The decay rate for $E[\phi](\Si_u)$ is $u_+^{-\ga_0}$ in view of Proposition \ref{ephidecay0} with $\ga_0{ <} \min\{2, p-1\}$. We can upgrade this decay rate to $u_+^{-\ga}$ for any $\ga<2$ by using the second order derivative of the solution.
The main task of this section is to show that
\begin{Prop}
\label{decayezphi}
 If $\frac{1+\sqrt{17}}{2}<p<5$, for any $1<\ga<2$ we have
  \begin{align*}
  E[Z^{\leq 1}\phi](\Sigma_u)\lesssim u_+^{-\ga}\cE_1^{ C}\cE_2.
\end{align*}
\end{Prop}
Note that $\phi$ verifies the same equation for $Z\phi$
$$\Box_g(Z\phi)=p|\phi|^{p-1}Z\phi$$
 up to a constant. It suffices to prove the above decay estimate for $Z\phi$.
In view of the above equation for $Z\phi$, we need to control the nonlinearity $|\phi|^{p-1}$. Let's first recall the necessary trace type Sobolev inequalities. On $\Si_u$,   the trace theorem (see for example  Proposition 3.2.3 of \cite{Klainerman93:Mink}) implies that
\begin{align}
\label{eq:trace}
r\|\psi(r, u, \om)\|_{L_{\om}^4}^{2}\les E[\psi](\Si_u).
\end{align}
Sobolev embedding on the unit sphere then gives the pointwise decay estimate for the solution $\phi$ if we also have the energy flux decay for  $\Om_{ij}\phi$. Such pointwise decay estimates are sufficiently strong when $r\les u_+$. To transfer the  decay in $u$ to decay in $r$ on $\H_u$, we make use of the $r$-weighted energy estimate.
 \begin{Lem}
 \label{betterrl}
On the out going null hypersurface $\mathcal{H}_u$, it holds that
\begin{equation*}
\|r\psi(r_2, u, \om)\|_{L_\om^{q_1}}^{q_1} \lesssim  \|r\psi(r_1, u, \om)\|_{L_\om^{q_1}}^{q_1}+(r_1^{\frac{q_1}{2}-\ga}+r_2^{\frac{q_1}{2}-\ga})E[\psi](\Si_u)^{\frac{q_1-2}{2}}\cdot \int_{  \H_u}r^{\ga} |L(r\psi)|^2dvd\om
\end{equation*}
for all $r_2\geq r_1 \geq R$,\quad $2\leq q_1\leq 4$, $ {q_1}\neq2\ga$ and $\ga\in \mathbb{R}$.
\end{Lem}

\begin{proof}
First we have
\begin{align*}
&\int_{|\om|=1}|r\psi|^{q_1} d\om|_{r=r_2} =\int_{|\om|=1}|r\psi|^{q_1} d\om|_{r=r_1}+\int_{r_1}^{r_2}\int_{|\om|=1}q_1 L(r\psi) (r\psi)^{q_1-1}dvd\om\\
&\les \int_{|\om|=1}|r\psi|^{q_1} d\om|_{r=r_1}+\left(\int_{r_1}^{r_2}\int_{|\om|=1} r^{\ga}|L(r\psi)|^2 dvd\om \right)^{\f12}\cdot \left(\int_{r_1}^{r_2}\int_{|\om|=1}r^{2q_1-2-\ga} |\psi|^{2q_1-2}dvd\om \right)^{\f12}.
\end{align*}
By using the Gagliardo-Nirenberg inequality on the unit sphere, we can estimate that
$$
\int_{|\om|=1}|\psi|^{2q_1-2}d\omega \les \|(\psi, \pa_{\om}\psi)\|_{L_\om^2}^{2q_1-2-q_1} \int_{|\om|=1}|\psi|^{q_1} d\omega.
$$
Fix $r_1$ and take supreme in term of $r_2$, we in particular obtain that
\begin{align*}
&\sup\limits_{r_1\leq r\leq r_2}\int_{|\om|=1}|r\psi|^{ q_1} d\om
\les \int_{|\om|=1}|r\psi|^{ q_1} d\om|_{r=r_1}+\\&\left(\int_{r_1}^{r_2}\int_{|\om|=1} r^{\ga}|L(r\psi)|^2 dvd\om \right)^{\f12}\cdot \left(\sup\limits_{r_1\leq r\leq r_2}\int_{|\om|=1}|r\psi|^{ q_1} d\om\right)^{\f12} \cdot \left(\int_{r_1}^{r_2} r^{{ q_1}-2-\ga} \|(\psi, \pa_{\om}\psi)\|_{L_\om^2}^{{ q_1}-2} dv \right)^{\f12},
\end{align*}
from which we conclude that
\begin{align*}
 &\int_{|\om|=1}|r\psi|^{q_1} d\om |_{r=r_2} \\
&\les \int_{|\om|=1}|r\psi|^{q_1} d\om|_{r=r_1}+ \int_{ \H_u} r^{\ga}|L(r\psi)|^2 dvd\om   \cdot   \int_{r_1}^{r_2} r^{q_1-2-\ga} \|(\psi, \pa_{\om}\psi)\|_{L_\om^2}^{q_1-2} dv\\
&\les  \int_{|\om|=1}|r\psi|^{q_1} d\om|_{r=r_1}+ \int_{ \H_u} r^{\ga}|L(r\psi)|^2 dvd\om   \cdot    \| \|(\psi, \pa_{\om}\psi)\|_{L_\om^2}^{q_1-2}  \|_{L_v^{\frac{2}{q-2}}} \|r^{q_1-2-\ga}\|_{L_v^{\frac{2}{4-q_1}}([r_1, r_2])}.
\end{align*}
The Lemma then follows as
\begin{align*}
 \| \|(\psi, \pa_{\om}\psi)\|_{L_\om^2}^{q_1-2}  \|_{L_v^{\frac{2}{q_1-2}}}&\les \int_{\H_u}|\psi|^2+r^2|\nabb\psi|^2dvd\om \les E[\psi](\Si_u),\\
 \|r^{q_1-2-\ga}\|_{L_v^{\frac{2}{4-q_1}}([r_1, r_2])}&\les r_1^{\frac{q_1}{2}-\ga}+r_2^{\frac{q_1}{2}-\ga}.
\end{align*}

\end{proof}

As a corollary, we derive the following key bound for the nonlinearity.
\begin{Cor}
\label{r3}
For $\frac{1+\sqrt{17}}{2}<p<5$, there exist $s>1$ and $\epsilon_1>0$ such that on the  hypersurface $\Si_u$, the solution $\phi$ satisfies the following decay estimate
  \begin{equation*}
\Big(\int_{|\om|=1}|\phi(r, u, \om )|^{2(p-1)s}d\omega\Big)^{\frac{1}{s}}
\lesssim r^{-3-\epsilon_1}u_{+}^{-1-\epsilon_1} \cE_1^{p-1}.
\end{equation*}
Here with out loss of generality, we may assume $\cE_1\geq 1$.
\end{Cor}
\begin{proof}
In view of the proof for Proposition \ref{ephidecay0}, for all $N>\ga_0$ and $1<\ga_0<\min\{2, p-1\}$ we recall that
\begin{align*}
&\int_{\H_u}r^{\ga_0}|L(r\phi)|^2dvd\om \les E[\phi](\tSi_{0})\les 1,\\
&E[\phi, p](\Sigma_u)\lesssim
  u_+^{-\gamma_0}  +u_+^{-N} \cE_1\les u_+^{-\ga_0}\cE_1. 
\end{align*}
 In particular for all $1\leq q_1\leq 4$ the trace theorem shows that
\begin{align*}
\|\phi(r, u, \om)\|_{L_{\om}^{q_1}}^2\les \|\phi(r, u, \om)\|_{L_{\om}^4}^{2}&\les r^{-1} E[\phi](\Si_u)\les r^{-1} u_+^{-\gamma_0}   \cE_1.
\end{align*}
Then the previous lemma \ref{betterrl} implies that
 \begin{align*}
\|r\phi(r_2, u, \om)\|_{L_\om^{q_1}}^{q_1} &\lesssim  \|r\phi(r_1, u, \om)\|_{L_\om^{q_1}}^{q_1}+(r_1^{\frac{q_1}{2}-\ga_0}+r_2^{\frac{q_1}{2}-\ga_0})E[\phi](\Si_u)^{\frac{q_1-2}{2}} \\
&\les r_1^{\frac{q_1}{2}}  \big(E[ \phi](\Si_u)\big)^{\frac{q_1}{2}}  +(r_1^{\frac{q_1}{2}-\ga_0}+r_2^{\frac{q_1}{2}-\ga_0})\big(E[ \phi](\Si_u)\big)^{\frac{q_1-2}{2}}  .
\end{align*}
Setting $r_1=R$, we in particular derive that
\begin{align*}
\|\phi(r, u, \om)\|_{L_\om^{q_1}}
&\les r^{-1}   \big(E[ \phi](\Si_{u})\big)^{\frac{1}{2}} +\big(  r^{-\f12-\frac{\ga_0}{q_1}}+r^{-1} \big)  \big(E[ \phi](\Si_{u})\big)^{\frac{q_1-2}{2q_1}}\\
&\les  \big(  r^{-\f12-\frac{\ga_0}{q_1}}+r^{-1} \big)  \big( u_+^{-\ga_0}\cE_1 \big)^{\frac{q_1-2}{2q_1}}
\end{align*}
for all $2\leq  q_1\leq 4$, $ {q_1}\neq2\ga_0$. Here we recall that $E[\phi](\Si_u)\les E[\phi, p](\tSi_0)\les 1$.
Now for the case when  $\frac{1+\sqrt{17}}{2}<p<3$, we can choose $s>1$ such that  $q_1=2(p-1)s\in (2, 4)$ and $1<\ga_0<p-1$,  then $q_1>2(p-1)>2\ga_0$. Using H\"older's inequality on the unit sphere, on one hand have
\begin{align*}
\Big(\int_{|\om|=1}|\phi(r, u, \om)|^{2(p-1)s}d\omega\Big)^{\frac{1}{s}}
\lesssim \|\phi\|_{L_{\om}^{q_1}}^{2(p-1)}\les  r^{1-p-\frac{\ga_0}{s}}u_+^{-(p-1)\ga_0+\frac{\ga_0}{s}}\cE_1^{p-1-\frac{1}{s}}.
\end{align*}
This bound is sufficiently strong when $r\geq u_+$. On the other hand when $r\leq u_+$, directly using the above $L^4$ bound, we also have
\begin{align*}
\Big(\int_{|\om|=1}|\phi(r, u, \om)|^{2(p-1)s}d\omega\Big)^{\frac{1}{s}}
\lesssim \|\phi\|_{L_{\om}^{4}}^{2(p-1)}\les  r^{1-p}u_+^{-(p-1)\ga_0 }\cE_1^{p-1}.
\end{align*}
Note that the total decay rate in $u$ and $r$ verifies
\begin{align*}
p-1+\ga_0(p-1)>4
\end{align*}
for $\ga_0$ sufficiently close to $p-1$ (the restriction that $\ga_0<p-1$) and $p>\frac{1+\sqrt{17}}{2}$.
Moreover, note that $2(p-1)>3$, $\ga_0(p-1)>1$. The lemma then follows for the case $\frac{1+\sqrt{17}}{2}<p<3$.

For the case when $3\leq p<5$, let $s>1$ such that $2(p-1)s>4$. Using the Gagliardo-Nirenberg inequality, similarly we can show that
\begin{align*}
\int_{|\om|=1}|\phi|^{2(p-1)s}d\omega &\les \|(\phi, \pa_{\om}\phi)\|_{L_\om^2}^{2(p-1)s-4} \cdot \int_{|\om|=1}|\phi|^{4} d\omega\\
&\les (r^{ -1}E[Z^{\leq 1}\phi](\Si_u))^{ (p-1)s-2}  \min\{r^{-2}u_+^{-2\ga_0}\cE_1^2, \cE_1 r^{-2-\ga_0}u_+^{-\ga_0}\}\\
&\les r^{-(p-1)s}u_+^{-2\ga_0}\min\{1, u_+^{\ga_0} r^{-\ga_0}\} \cE_1^{(p-1)s}.
\end{align*}
Choose $\ga_0$ and $s$ such that
$$\frac{2}{p-1}\leq 1<\ga_0< 2,\quad  1<s <\ga_0<2.$$
Then the total decay rate divided by $s$ satisfies
\begin{align*}
p-1+2\frac{\ga_0}{s}>p+1\geq 4.
\end{align*}
Moreover since $2\frac{\ga_0}{s}>1$ and $p-1+\frac{\ga_0}{s}>3$, we then conclude that there exists $\ep_1>0$ such that
 \begin{align*}
\left(\int_{|\om|=1}|\phi|^{2(p-1)s}d\omega \right)^{\frac{1}{s}}
&\les r^{-(p-1)}u_+^{-2\frac{\ga_0}{s}}\min\{1, u_+^{\frac{\ga_0}{s}} r^{-\frac{\ga_0}{s}}\} \cE_1^{p-1}\\
&\les r^{-3-\ep_1} u_+^{-1-\ep_1} \cE_1^{p-1}.
\end{align*}
\end{proof}
To prove Proposition \ref{decayezphi}, similar to the proof for the energy flux decay estimate for $\phi$, we need to establish a $r$-weighted energy estimate for $Z\phi$.
 \begin{Prop}
 \label{rwzphi}
If $\frac{1+\sqrt{17}}{2}<p<5$, for any $1<\gamma<2$ and $k=0, 1$, we have the following $r$-weighted estimate for $Z^k\phi$
  \begin{equation*}
  \begin{split}
  &\iint_{\mathcal{D}_{u_1, u_2}\cap\{r\geq R\}} r^{\gamma-3}\Big(
  |L(rZ^k\phi)|^2
  +|Z^k\phi|^2
  +|r\slashed\nabla Z^k\phi|^2\Big) d\vol
  +\int_{\mathcal{H}_{u_2}}r^{\gamma}|L(rZ^k\phi)|^2dvd\omega
  \\&
  \lesssim\int_{\mathcal{H}_{u_1}}
  r^{\gamma}|L(rZ^k\phi)|^2dvd\omega+ \cE_1^{ C} E[Z^k\phi](\Sigma_{u_1}).
  \end{split}
\end{equation*}
\end{Prop}

\begin{proof}
The proof goes similar to the proof for the Proposition \ref{divsk1} once we have controlled the nonlinearity. Let $ f=1$ 
and $\eta(r)$, $m(r)$, $q(r)$ be the same as those in the proof for the Proposition \ref{divsk1}. Let $R$ be sufficiently large depending on $\ga$ and $\epsilon $ sufficiently small depending on $R$ such that when $r\geq R$ it holds that
\begin{align*}
2-\ga+O(\mu)>0,\quad \ga (1-\frac{\ga}{2})^2+O(\mu)>0, \quad -f'= 0.
\end{align*}
Denote $\psi=Z^k\phi$.
In view of Lemma \ref{rpw}, when $r\geq R$, we have the lower bound
  \begin{align*}
    Q[\psi,\eta X, \eta q, \eta m]=Q[Z^{ k}\phi,  X,  q,  m]
    & \gtrsim r^{\gamma-3}\big(|L(r \psi)|^2+|r\slashed\nabla \psi|^2+|\psi|^2\big).
\end{align*}
For the nonlinear term, we can bound that
\begin{equation*}
  \begin{split}
  |\Box_g \psi (X \psi+q \psi)|&\les f r^{\gamma-1}|\phi|^{p-1} |\psi| |L(r \psi)|
  \\&
  \les   u_+^{-1-\epsilon_1} r^{\gamma-2}|L(r \psi)|^2
  +  u_+^{1+\epsilon_1}r^{\gamma}|\phi|^{2(p-1)} |\psi|^2.
  \end{split}
\end{equation*}
Here $\ep_1 >0$ is the small positive constant in Corollary \ref{r3}.
 Now apply the energy identity \eqref{eq:energy:id} to the scalar field $\psi=Z^k\phi$, $k=0, 1$, the vector field $\eta X$, the function  $\eta q$ and 1-form $\eta m$ on the domain $\cD_{u_1, u_2}^{v_0}$ bounded by $\Si_{u_1}$, $\Si_{u_2}$ and $\Hb_{v_0}^{u_1, u_2}$. We obtain the energy identity
\begin{equation*}
\begin{split}
  &\iint_{\mathcal{D}_{u_1, u_2}^{v_0}}Q[\psi, \eta X, \eta q, \eta m]+\eta\Box_g \psi ( X \psi +q \psi) d\vol+\int_{\underline{\mathcal{H}}_{v_0}^{u_1,u_2}} \tilde{P}_{u}r^2dud\omega+\int_{\mathcal{H}_{u_2}^{v_0}}\tilde{P}_{v}r^2dvd\omega\\
  &=-\int_{\Sigma_{u_1}\cap\{r<R \}}\tilde{P}^{t}r^2drd\omega
  +\int_{\mathcal{H}_{u_1}^{v}}\tilde{P}_{v}r^2dvd\omega
+\int_{\Sigma_{u_2}\cap\{ r<R \}}\tilde{P}^{t}r^2drd\omega,
\end{split}
\end{equation*}
in which  $\tilde{P}$ is short for $P[\psi, \eta X, \eta q, \eta m, 0]$.
Similar to the proof for Proposition \ref{divsk1}, the lower order terms (for example  $-\frac{1}{2} L(\eta f (1-\mu)r^{\ga+1}|\psi|^2)$  on the out going null hypersurface $\H_{u}$) are canceled by doing integration by parts. Here we emphasize that the boundary of $\cD_{u_1, u_2}^{v_0}$ is closed and $\eta=0$ when $r\leq 4M$. Now in view of Lemma \ref{rpw}, we move all the terms restricted to region $\{r\geq R\}$ with positive sign to the left hand side and bound the terms on the region $\{r\leq R\}$ by the integrated local energy estimate of Proposition \ref{H2}. We then derive that
  \begin{align*}
  I:&=\iint_{\mathcal{D}_{u_1, u_2}^{v_0}\cap\{r\geq R\}}r^{\gamma-3}\Big(
  |L(r\psi)|^2
  +|\psi|^2
  +|r\slashed\nabla \psi|^2\Big) d\vol
  +\int_{\mathcal{H}_{u_2}^{v_0}}r^{\gamma}|L(r\psi)|^2dvd\omega
  \\&
  \lesssim \iint_{\mathcal{D}_{u_1, u_2}^{v_0} }  \eta u_+^{-1-\epsilon_1} r^{\gamma-2}|L(r \psi)|^2
  +\eta u_+^{1+\epsilon_1}r^{\gamma}|\phi|^{2(p-1)} |\psi|^2 d\vol +\int_{\mathcal{H}_{u_2}^{v_0}}r^{\gamma}|L(r\psi)|^2dvd\omega\\
  &\quad +\int_{(\Si_{u_1}\cup \Si_{u_2})\cap\{4M\leq r<R\}} |\pa \psi|^2+|\psi|^2dx+\iint_{\mathcal{D}_{u_1, u_2}^{v_0}\cap\{4M\leq r\leq R\}} |\pa \psi|^2+|\psi|^2d\vol\\
  &
  \lesssim \iint_{\mathcal{D}_{u_1, u_2}^{v_0} }  \eta u_+^{-1-\epsilon_1} r^{\gamma-2}|L(r \psi)|^2
  +\eta u_+^{1+\epsilon_1}r^{\gamma}|\phi|^{2(p-1)} |\psi|^2 d\vol   +\int_{\mathcal{H}_{u_2}^{v_0}}r^{\gamma}|L(r\psi)|^2dvd\omega+E[\psi](\Si_{u_1}).
\end{align*}
The first term on the right hand side will be absorbed by using Gronwall's inequality. For the second term involving the nonlinearity, using Corollary \ref{r3} and Sobolev embedding on the unit sphere, we can estimate that
\begin{equation*}
  \begin{split}
  \int_{|\om|=1}|\phi|^{2(p-1)}|\psi|^2d\omega
  &\leq \Big(\int_{|\om|=1}|\phi|^{2(p-1)s}d\omega\Big)^{\frac{1}{s}}
  \Big(\int_{|\om|=1}|\psi|^{2s'}d\omega\Big)^{\frac{1}{s'}}
  \\&
  \lesssim
  \Big(\int_{|\om|=1}|\phi|^{2(p-1)s}d\omega\Big)^{\frac{1}{s}}
  \int_{|\om|=1}|\psi|^2+|r\slashed\nabla \psi|^2d\omega
  \\&
  \lesssim r^{-3-\epsilon_1}u_+^{-1-\epsilon_1} \cE_1^{p-1}\int_{|\om|=1}| \psi|^2+|r\slashed\nabla \psi|^2d\omega.
  \end{split}
\end{equation*}
Since $\ga<2$ and $\eta$ vanishes when $r\leq 4M$, we thus can show that
\begin{align*}
&\iint_{\mathcal{D}_{u_1, u_2}^{v_0} }
  \eta u_+^{1+\epsilon_1}r^{\gamma}|\phi|^{2(p-1)} |\psi|^2 d\vol \les \cE_1^{p-1}\iint_{\mathcal{D}_{u_1, u_2}^{v_0} \cap\{r\geq 4M\} }r^{\gamma-3-\epsilon_1} (|\psi|^2+|r\slashed\nabla \psi|^2) d\vol\\ \leq& \cE_1^{p-1}\left(\iint_{\mathcal{D}_{u_1, u_2}^{v_0} \cap\{r\geq 4M\} }r^{-4} (|\psi|^2+|r\slashed\nabla \psi|^2) d\vol\right)^{\frac{\epsilon_1}{\gamma+1}}\left( \iint_{\mathcal{D}_{u_1, u_2}^{v_0} \cap\{r\geq 4M\} }r^{\gamma-3} (|\psi|^2+|r\slashed\nabla \psi|^2) d\vol\right)^{1-\frac{\epsilon_1}{\gamma+1}}\\ \les& \cE_1^{p-1}E[\psi](\Si_{u_1})^{\frac{\epsilon_1}{\gamma+1}}(E[\psi](\Si_{u_1})+I)^{1-\frac{\epsilon_1}{\gamma+1}}.
\end{align*}
The last step follows from the integrated local energy estimate of Proposition \ref{H2}.
By using Gronwall's inequality, the $r$-weighted energy estimate for $Z^{k}\phi$ then follows by letting  $v_0\rightarrow \infty$.

\end{proof}

 Now we are able to prove Proposition \ref{decayezphi}.
\begin{proof}[\textbf{Proof of Proposition \ref{decayezphi}}]
The proof goes similar to that for Proposition \ref{ephidecay0} once we have the above $r$-weighted energy estimate for $Z\phi$ and the Strichartz estimate for $Z^2\phi$. The differences are that the implicit constant also relies on the first order energy $\cE_1$ and the Strichartz estimate for the second order derivative of the solution may grow in time instead of being uniformly bounded. We will elaborate the main differences in the argument and skip those steps similar to the proof of Proposition \ref{ephidecay0}.

To obtain a bound for the integral of the energy flux $E[Z\phi](\Si_u)$ from time $u_1$ to $u_2$, it is important to control the integral near the photon sphere $r=3M$, which can not be bounded by the integrated local energy estimate due to the degeneracy on the photon sphere. This degeneracy can be compensated by higher order energy via Lemma \ref{lp}. More precisely, by using Lemma \ref{lp} together with the Strichartz estimate of Proposition \ref{H2}, we can show that
\begin{align*}
 &\int_{u_1}^{u_2}\int_{ |r^*|\leq \f12}|(\partial_{\tv}, \slashed\nabla)Z\phi|^2 d\vol\\
 &  \lesssim  \int_{u_1}^{u_2}\int_{ |r^*|\leq \f12 }|\ln u_{+}|^2|\ln |r^{*}||^{-2}|(\partial_{\tv},\slashed\nabla)Z\phi|^2 dx
 +\int_{u_1}^{u_2} u_{+}^{- 5N/6} \|( \partial_{\tv},\slashed\nabla) Z\phi\|_{L_x^{12}( |r^*|\leq \f12)}^2du\\
 &\les (\ln (u_2)_+)^2 \|Z\phi\|_{LE(\cD_{u_1, u_2})}+\int_{\tv_1}^{\tv_2} \tv_{+}^{- 5N/6} \| Z^2 \phi\|_{L_x^{12}}^2 d\tv \\
 &\les (\ln (u_2)_+)^2 E[Z\phi](\Si_{u_1})+(\tv_1)_+^{ \frac{1}{2}-\frac{N}{3}}\|\tv_+^{- \frac{N}{4}}Z^2\phi\|_{L_{\tv}^4 L_x^{12}(\tcD_{\tv_1, \tv_2})}^2.
\end{align*}
Here $\tv_i$, $i=1, 2$ are the parameters such that $\tSi_{\tv_i}\cap\{r\leq R\}=\Si_{u_i}\cap \{r\leq R\}$.
Now in view of the Strichartz estimate of Lemma \ref{Z2phi}, we in particular have that
\begin{align*}
\| Z^2\phi\|_{L_{\tv}^4 L_x^{12}(\tcD_{\tv_1, \tv_2})}^4\les (\tv_2)_+^{ 4} \cE_2^2
\end{align*}
for all $\tv_1\leq \tv_2$. Hence for $ N>4$, we can show that
\begin{align*}
\|\tv_+^{- \frac{N}{4}}Z^2\phi\|_{L_{\tv}^4 L_x^{12}(\tcD_{\tv_1, \tv_2})}^4&=\int_{\tv_1}^{\tv_2} \tv_+^{- N} d \|Z^2\phi\|_{L_{\tv}^4 L_x^{12}(\tcD_{\tv_1, \tv })}^4\\
&=(\tv_2)_+^{- N} \|Z^2\phi\|_{L_{\tv}^4 L_x^{12}(\tcD_{\tv_1, \tv_2 })}^4+N \int_{\tv_1}^{\tv_2} \tv_+^{- N-1}  \|Z^2\phi\|_{L_{\tv}^4 L_x^{12}(\tcD_{\tv_1, \tv })}^4d\tv \\
&\les (\tv_2)_+^{ 4-N} \cE_2^2+N \int_{\tv_1}^{\tv_2} \tv_+^{ -N+3}  \cE_2^2 d\tv\\
&\les (\tv_1)_+^{ 4-N}\cE_2^2.
\end{align*}
We therefore have shown that
\begin{align*}
 \int_{u_1}^{u_2}\int_{ |r^*|\leq \f12}|(\partial_{\tv}, \slashed\nabla)Z\phi|^2 d\vol
  &\les (\ln (u_2)_+)^2 E[Z\phi](\Si_{u_1})+(u_1)_+^{ \frac{5}{2}-\frac{5N}{6}}\cE_2 .
\end{align*}
Here note that $(u_1)_+\les (\tv_1)_+$ and $ N>4$.

Then similar to the proof for Proposition \ref{ephidecay0}, we can obtain
\begin{equation*}
\begin{split}
  \int_{u_1}^{u_2}E[Z\phi](\Sigma_u)du
     \les  (\ln (u_2)_+)^2 E[Z\phi](\Si_{u_1})+(u_1)_+^{ \frac{5}{2}-\frac{5N}{6}}\cE_2 +\int_{\mathcal{H}_{u_1}}
  r|L(rZ\phi)|^2dvd\omega+\cE_1^{C}E[Z\phi](\Sigma_{u_1}).
  \end{split}
\end{equation*}
Now the $r$-weighted energy estimate of Proposition \ref{rwzphi} implies that there is a dyadic sequence  $\{u_k\}_{3}^{\infty}$ such that
\begin{equation*}
\begin{split}
  \int_{\mathcal{H}_{u}}r^{\gamma}|L(rZ\phi)|^2 dvd\omega
 \lesssim \cE_1^{ C} ,\quad \int_{\mathcal{H}_{u_k}}r^{\gamma-1}|L(rZ\phi)|^2dvd\omega\lesssim u_k^{-1} \cE_1^{ C}.
\end{split}
\end{equation*}
In particular interpolation shows that
\begin{equation*}
\int_{\mathcal{H}_{u_k}}r|L(r Z\phi)|^2dvd\omega\lesssim u_k^{1-\gamma} \cE_1^{ C},\ k\geq 3.
\end{equation*}
Using the energy estimate of  Proposition \ref{H2}, we first can get that  (for $u>1$)
\begin{equation*}
\begin{split}
  \frac{1}{2}u E[Z\phi](\Sigma_{u})\les \int_{\frac{u}{2}}^u E[Z\phi](\Si_{u'})du'
  \lesssim
  (\ln u_{+})^2\cE_1+u_+^{ \frac{5}{2}-\frac{5N}{6}}\cE_2+\cE_1^{ C}\lesssim
  (\ln u_{+})^2\cE_1^{ C}+u_+^{ \frac{5}{2}-\frac{5N}{6}}\cE_2 .
  \end{split}
\end{equation*}
Here recall that we have assumed that $\cE_1\geq 1$.
In particular we derive a weak decay estimate for the energy flux
\begin{equation*}
  E[Z\phi](\Sigma_u)\lesssim u_+^{-1} (\ln u_+)^2 \cE_1^{ C}+u_+^{ \frac{3}{2}-\frac{5N}{6}}\cE_2,\quad \forall u\geq -\frac{R^*}{2}.
\end{equation*}
Based on this weak decay estimate, we apply the above energy inequality to the sequence $u_k$
 \begin{align*}
  &(u_{k+1}-u_{k})E[Z\phi](\Sigma_{u_{k+1}})\lesssim \int_{u_{k}}^{u_{k+1}}E[Z\phi ](\Sigma_{u})du \\
  &\lesssim \ln(u_{k+1})^2E[Z\phi ](\Sigma_{u_k})+
  \int_{\mathcal{H}_{u_k}}r|L(rZ\phi)|^2dvd\omega
  +u_k^{ \frac{5}{2}-\frac{5N}{6}}\cE_2+\cE_1^{ C} E[Z\phi](\Si_{u_k})
  \\&
  \lesssim u_{k}^{1-\gamma}\cE_1^{ C}+u_{k}^{ \frac{5}{2}-\frac{5N}{6}}\cE_2+\ln(u_{k+1})^2 \cE_1^{ C}E[Z\phi ](\Sigma_{u_k}).
\end{align*}
Since the sequence $u_k$ is dyadic, for general $u\in [u_k, u_{k+1}]$, using the above weak decay estimate for $E[Z\phi](\Si_u)$, we can show that
\begin{align*}
E[Z\phi](\Si_u)&\les u_{+}^{-\gamma}\cE_1^{ C}+u_{+}^{ \frac{3}{2}-\frac{5N}{6}}\cE_2+u_+^{-1} \ln(u_{k+1})^2 \cE_1^{ C} E[Z\phi ](\Sigma_{u_k})\\
 &\les u_+^{-\ga}\cE_1^{ C}+u_+^{ \frac{3}{2}-\frac{5N}{6}}\cE_1^{ C}\cE_2\les u_+^{-\ga}\cE_1^{ C}\cE_2.
\end{align*}
We thus finished the proof for Proposition \ref{decayezphi}.

\end{proof}

\section{Pointwise decay estimates for the solution}
 Now we are able to derive the desired pointwise decay estimate for the solution. First of all, using the trace theorem \eqref{eq:trace}, we have
 \begin{align*}
 \|Z^k\phi(u, r, \om)\|_{L_{\om}^4}\les r^{-\f12 }(E[Z^k\phi](\Si_u))^{\f12},\quad k\leq 1.
 \end{align*}
 As $Z$ can be the angular momentum vector fields $\Om_{ij}$, in view of the improved energy flux decay estimates of Proposition \ref{decayezphi}, Sobolev embedding on the unit sphere then leads to the rough decay estimate
 \begin{align*}
 |\phi(u, r, \om)|\les \sum\limits_{k\leq 1}{ \sum\limits_{i,j}}\|\Om_{i, j}^k\phi(u, r, \om)\|_{L_{\om}^4}\les r^{-\f12 }u_+^{-\frac{\ga}{2}} \cE_1^{ C}\sqrt{\cE_2}.
 \end{align*}
 This decay estimate is sufficiently strong in the region $\{r\leq R\}$, where the retarded time $u$ is comparable to the physical time $t$ or $\tv$.

 To improve the decay estimate for the solution in the far away region near the future null infinity, we make use of the $r$-weighted energy estimate of Proposition \ref{rwzphi} and Lemma  \ref{betterrl}. First from the $r$-weighted energy estimate for $Z^k\phi$, $k\leq 1$, we conclude that
 \begin{align*}
 \int_{\H_u}r^{\ga}|L(rZ^k\phi)|^2 dvd\om\les \cE_1^{ C}
 \end{align*}
 by setting $u_1=-\frac{R^*}{2}$ and $u>u_1.$
Then from Lemma \ref{betterrl}, for all
$$1<\ga<2, \quad 2<q_1<2\ga,\quad r_2=r,\quad r_1=u_+,\quad  r\geq u_+\geq R $$
 we have
 \begin{align*}
\|r Z^k\phi(r, u, \om)\|_{L_\om^{q_1}}^{q_1} &\lesssim u_+^{\frac{q_1}{2}} \|u_+^{\f12}  Z^k\phi(u_+, u, \om)\|_{L_\om^{q_1}}^{q_1}+\cE_1^{ C} u_+^{\frac{q_1}{2}-\ga}  E[Z^k\phi](\Si_u)^{\frac{q_1-2}{2}}  \\
&\les u_+^{\frac{q_1}{2}} \| u_+^{\f12} Z^k\phi(u_+, u, \om)\|_{L_\om^{4}}^{q_1}+\cE_1^{ C}  u_+^{\frac{q_1}{2}-\ga}  E[Z^k\phi](\Si_u)^{\frac{q_1-2}{2}} \\
&\les u_+^{\frac{q_1}{2}} E[Z^k\phi](\Si_u)^{\frac{q_1}{2}}+\cE_1^{ C}  u_+^{\frac{q_1}{2}-\ga}  E[Z^k\phi](\Si_u)^{\frac{q_1-2}{2}}  .
\end{align*}
By using Sobolev embedding, we derive that
\begin{align*}
|\phi(u, r, \om)|&\les r^{-1} u_+^{\frac{1-\ga}{2}} \cE_1^{ C}\sqrt{\cE_2}+\cE_1^{ C}u_+^{\f12-\frac{\ga}{q_1}} u_+^{-\ga \frac{q_1-2}{2q_1}}  (\cE_1^{ C}\sqrt{\cE_2})^{\frac{q_1-2}{q_1}}\\
&\les r^{-1} u_+^{\frac{1-\ga}{2}} \cE_1^{ C}\sqrt{\cE_2}.
\end{align*}
We thus finished the proof for the main Theorem.


\begin{thebibliography}{10}

\bibitem{Soffer99:NLS:Schwarz}
I.~\L aba and A.~Soffer.
\newblock Global existence and scattering for the nonlinear {S}chr\"{o}dinger
  equation on {S}chwarzschild manifolds.
\newblock {\em Helv. Phys. Acta}, 72(4):274--294, 1999.

\bibitem{Blue16:Max:Schwarz}
L.~Andersson, T.~B\"{a}ckdahl, and P.~Blue.
\newblock Decay of solutions to the {M}axwell equation on the {S}chwarzschild
  background.
\newblock {\em Classical Quantum Gravity}, 33(8):085010, 20, 2016.

\bibitem{Blue:hiddensymmetry}
L.~Andersson and P.~Blue.
\newblock Hidden symmetries and decay for the wave equation on the {K}err
  spacetime.
\newblock {\em Ann. of Math. (2)}, 182(3):787--853, 2015.

\bibitem{Blue:energybound:smallKerr}
L.~Andersson and P.~Blue.
\newblock Uniform energy bound and asymptotics for the {M}axwell field on a
  slowly rotating {K}err black hole exterior.
\newblock {\em J. Hyperbolic Differ. Equ.}, 12(4):689--743, 2015.

\bibitem{Blue13:decay:trap}
L.~Andersson, P.~Blue, and J.~Nicolas.
\newblock A decay estimate for a wave equation with trapping and a complex
  potential.
\newblock {\em Int. Math. Res. Not. IMRN}, (3):548--561, 2013.

\bibitem{Blue20:linGra:Schwarz}
L.~Andersson, P.~Blue, and J.~Wang.
\newblock Morawetz estimate for linearized gravity in {S}chwarzschild.
\newblock {\em Ann. Henri Poincar\'{e}}, 21(3):761--813, 2020.

\bibitem{Stefanos18:Latetime}
Y.~Angelopoulos, S.~Aretakis, and D.~Gajic.
\newblock Late-time asymptotics for the wave equation on spherically symmetric,
  stationary spacetimes.
\newblock {\em Adv. Math.}, 323:529--621, 2018.

\bibitem{stefanos18:vector}
Y.~Angelopoulos, S.~Aretakis, and D.~Gajic.
\newblock A vector field approach to almost-sharp decay for the wave equation
  on spherically symmetric, stationary spacetimes.
\newblock {\em Ann. PDE}, 4(2):Paper No. 15, 120, 2018.

\bibitem{Blue08:Max:Schwarz}
P.~Blue.
\newblock Decay of the {M}axwell field on the {S}chwarzschild manifold.
\newblock {\em J. Hyperbolic Differ. Equ.}, 5(4):807--856, 2008.

\bibitem{Blue03:ILD:Schwarz}
P.~Blue and A.~Soffer.
\newblock Semilinear wave equations on the {S}chwarzschild manifold. {I}.
  {L}ocal decay estimates.
\newblock {\em Adv. Differential Equations}, 8(5):595--614, 2003.

\bibitem{Blue05:spin2:Schwarz}
P.~Blue and A.~Soffer.
\newblock The wave equation on the {S}chwarzschild metric. {II}. {L}ocal decay
  for the spin-2 {R}egge-{W}heeler equation.
\newblock {\em J. Math. Phys.}, 46(1):012502, 9, 2005.

\bibitem{Blue07:Mor:NLW:Sch}
P.~Blue and A.~Soffer.
\newblock A space-time integral estimate for a large data semi-linear wave
  equation on the {S}chwarzschild manifold.
\newblock {\em Lett. Math. Phys.}, 81(3):227--238, 2007.

\bibitem{Blue06:semiNLW:Schwarz}
P.~Blue and J.~Sterbenz.
\newblock Uniform decay of local energy and the semi-linear wave equation on
  {S}chwarzschild space.
\newblock {\em Comm. Math. Phys.}, 268(2):481--504, 2006.

\bibitem{Klainerman93:Mink}
D.~Christodoulou and S.~Klainerman.
\newblock {\em The global nonlinear stability of the Minkowski space},
  volume~41 of {\em Princeton Mathematical Series}.
\newblock Princeton University Press, Princeton, NJ, 1993.

\bibitem{Shatah:YM:generalManifold}
P.~Chru\'sciel and J.~Shatah.
\newblock Global existence of solutions of the {Y}ang-{M}ills equations on
  globally hyperbolic four-dimensional {L}orentzian manifolds.
\newblock {\em Asian J. Math.}, 1(3):530--548, 1997.

\bibitem{Igor19:decay:Teuko:smKerr}
M.~Dafermos, G.~Holzegel, and I.~Rodnianski.
\newblock Boundedness and decay for the {T}eukolsky equation on {K}err
  spacetimes {I}: {T}he case {$|a|\ll M$}.
\newblock {\em Ann. PDE}, 5(1):Paper No. 2, 118, 2019.

\bibitem{Igor19:linStab:Schwarz}
M.~Dafermos, G.~Holzegel, and I.~Rodnianski.
\newblock The linear stability of the {S}chwarzschild solution to gravitational
  perturbations.
\newblock {\em Acta Math.}, 222(1):1--214, 2019.

\bibitem{Dafermos21:Stab:Schwarz}
M.~Dafermos, G.~Holzegel, I.~Rodnianski, and M.~Taylor.
\newblock {The non-linear stability of the Schwarzschild family of black
  holes}.
\newblock 2021.
\newblock ar{X}iv:2104.08222.

\bibitem{Dafermos22:NW:Kerr}
M.~Dafermos, G.~Holzegel, I.~Rodnianski, and M.~Taylor.
\newblock {Quasilinear wave equations on asymptotically flat spacetimes with
  applications to Kerr black holes}.
\newblock 2022.
\newblock ar{X}iv:2212.14093.

\bibitem{Igor07:note:Schwarz}
M.~Dafermos and I.~Rodnianski.
\newblock {A note on energy currents and decay for the wave equation on a
  Schwarzschild background }.
\newblock 2007.
\newblock ar{X}iv:0710.0171.

\bibitem{dr3}
M.~Dafermos and I.~Rodnianski.
\newblock The redshift effect and radiation decay on black hole spacetimes.
\newblock {\em Comm. Pure Appl. Math.}, 62(7):859--919, 2009.

\bibitem{newapp}
M.~Dafermos and I.~Rodnianski.
\newblock A new physical-space approach to decay for the wave equation with
  applications to black hole spacetimes.
\newblock In {\em X{VI}th {I}nternational {C}ongress on {M}athematical
  {P}hysics}, pages 421--432. World Sci. Publ., Hackensack, NJ, 2010.

\bibitem{Igor11:bd:smallKerr}
M.~Dafermos and I.~Rodnianski.
\newblock A proof of the uniform boundedness of solutions to the wave equation
  on slowly rotating {K}err backgrounds.
\newblock {\em Invent. Math.}, 185(3):467--559, 2011.

\bibitem{Igor:wave:kerr:ful}
M.~Dafermos, I.~Rodnianski, and Y.~Shlapentokh-Rothman.
\newblock Decay for solutions of the wave equation on {K}err exterior
  spacetimes {III}: {T}he full subextremal case {$|a|<M$}.
\newblock {\em Ann. of Math. (2)}, 183(3):787--913, 2016.

\bibitem{Soffer11:Price:Schwarz}
R.~Donninger, W.~Schlag, and A.~Soffer.
\newblock A proof of {P}rice's law on {S}chwarzschild black hole manifolds for
  all angular momenta.
\newblock {\em Adv. Math.}, 226(1):484--540, 2011.

\bibitem{Jinhua22:NW:Schwarz}
S.~Huo and J.~Wang.
\newblock A large data theory for nonlinear wave on the {S}chwarzschild
  background.
\newblock {\em J. Differential Equations}, 335:120--200, 2022.

\bibitem{Wald87:Schwarz}
B.~Kay and R.~Wald.
\newblock Linear stability of {S}chwarzschild under perturbations which are
  nonvanishing on the bifurcation {$2$}-sphere.
\newblock {\em Classical Quantum Gravity}, 4(4):893--898, 1987.

\bibitem{Lai23:Strauss:Schwarz}
N.~Lai and Y.~Zhou.
\newblock Blow-up and lifespan estimate to a nonlinear wave equation in
  {S}chwarzschild spacetime.
\newblock {\em J. Math. Pures Appl. (9)}, 173:172--194, 2023.

\bibitem{Sogge14:Strauss:Kerr}
H.~Lindblad, J.~Metcalfe, C.~Sogge, M.~Tohaneanu, and C.~Wang.
\newblock The {S}trauss conjecture on {K}err black hole backgrounds.
\newblock {\em Math. Ann.}, 359(3-4):637--661, 2014.

\bibitem{Mihai18:NW:Schwarz}
H.~Lindblad and M.~Tohaneanu.
\newblock Global existence for quasilinear wave equations close to
  {S}chwarzschild.
\newblock {\em Comm. Partial Differential Equations}, 43(6):893--944, 2018.

\bibitem{Mihai20:ILE:Kerr:p}
H.~Lindblad and M.~Tohaneanu.
\newblock A local energy estimate for wave equations on metrics asymptotically
  close to {K}err.
\newblock {\em Ann. Henri Poincar\'{e}}, 21(11):3659--3726, 2020.

\bibitem{improvLuk}
J.~Luk.
\newblock Improved decay for solutions to the linear wave equation on a
  {S}chwarzschild black hole.
\newblock {\em Ann. Henri Poincar\'e}, 11(5):805--880, 2010.

\bibitem{Luk13:NW:Kerr}
J.~Luk.
\newblock The null condition and global existence for nonlinear wave equations
  on slowly rotating {K}err spacetimes.
\newblock {\em J. Eur. Math. Soc. (JEMS)}, 15(5):1629--1700, 2013.

\bibitem{Masy21:Price:spins}
S.~Ma and L.~Zhang.
\newblock Price's law for spin fields on a {S}chwarzschild background.
\newblock {\em Ann. PDE}, 8(2):Paper No. 25, 100, 2022.

\bibitem{Tataru10:Strich:Schwarz}
J.~Marzuola, J.~Metcalfe, D.~Tataru, and M.~Tohaneanu.
\newblock Strichartz estimates on {S}chwarzschild black hole backgrounds.
\newblock {\em Comm. Math. Phys.}, 293(1):37--83, 2010.

\bibitem{Igor05:sNW:BH}
M.Dafermos and I.~Rodnianski.
\newblock Small-amplitude nonlinear waves on a black hole background.
\newblock {\em J. Math. Pures Appl. (9)}, 84(9):1147--1172, 2005.

\bibitem{Tataru20:ILE:nontrap}
J.~Metcalfe, J.~Sterbenz, and D.~Tataru.
\newblock Local energy decay for scalar fields on time dependent non-trapping
  backgrounds.
\newblock {\em Amer. J. Math.}, 142(3):821--883, 2020.

\bibitem{Tataru12:Strichartz:par}
J.~Metcalfe and D.~Tataru.
\newblock Global parametrices and dispersive estimates for variable coefficient
  wave equations.
\newblock {\em Math. Ann.}, 353(4):1183--1237, 2012.

\bibitem{Tataru12:price}
J.~Metcalfe, D.~Tataru, and M.~Tohaneanu.
\newblock Price's law on nonstationary space-times.
\newblock {\em Adv. Math.}, 230(3):995--1028, 2012.

\bibitem{Morawetz68:KG}
C.~S. Morawetz.
\newblock Time decay for the nonlinear klein-gordon equations.
\newblock {\em Proc. Roy. Soc. Ser. A}, 306:291--296, 1968.

\bibitem{Georgios16:rp}
G.~Moschidis.
\newblock The {$r^p$}-weighted energy method of {D}afermos and {R}odnianski in
  general asymptotically flat spacetimes and applications.
\newblock {\em Ann. PDE}, 2(1):Art. 6, 194, 2016.




\bibitem{Pecher82:decay:3d}
H.~Pecher.
\newblock Decay of solutions of nonlinear wave equations in three space
  dimensions.
\newblock {\em J. Funct. Anal.}, 46(2):221--229, 1982.

\bibitem{price72:redshift}
R.~Price.
\newblock Nonspherical perturbations of relativistic gravitational collapse.
  {I}. {S}calar and gravitational perturbations.
\newblock {\em Phys. Rev. D (3)}, 5:2419--2438, 1972.

\bibitem{Ralston69:ILE}
J.~Ralston.
\newblock Solutions of the wave equation with localized energy.
\newblock {\em Comm. Pure Appl. Math.}, 22:807--823, 1969.

\bibitem{Jan15:GauB}
J.~Sbierski.
\newblock Characterisation of the energy of {G}aussian beams on {L}orentzian
  manifolds: with applications to black hole spacetimes.
\newblock {\em Anal. PDE}, 8(6):1379--1420, 2015.

\bibitem{Volker13:LSchw}
V.~Schlue.
\newblock Decay of linear waves on higher-dimensional {S}chwarzschild black
  holes.
\newblock {\em Anal. PDE}, 6(3):515--600, 2013.

\bibitem{Klainerman22:Stab:Kerr}
S.Klainerman and J.~Szeftel.
\newblock {Brief introduction to the nonlinear stability of Kerr}.
\newblock 2022.\\
\newblock ar{X}iv:2210.14400.

\bibitem{Tataru15:ILE:Maxwell:Schwarz}
J.~Sterbenz and D.~Tataru.
\newblock Local energy decay for {M}axwell fields {P}art {I}: {S}pherically
  symmetric black-hole backgrounds.
\newblock {\em Int. Math. Res. Not. IMRN}, (11):3298--3342, 2015.

\bibitem{Strauss:NLW:decay}
W.~Strauss.
\newblock Decay and asymptotics for {$ \Box u=F(u)$}.
\newblock {\em J. Functional Analysis}, 2:409--457, 1968.

\bibitem{Tataru13:localdecay}
D.~Tataru.
\newblock Local decay of waves on asymptotically flat stationary space-times.
\newblock {\em Amer. J. Math.}, 135(2):361--401, 2013.

\bibitem{Mihai11:ILE:Kerr}
D.~Tataru and M.~Tohaneanu.
\newblock A local energy estimate on {K}err black hole backgrounds.
\newblock {\em Int. Math. Res. Not. IMRN}, (2):248--292, 2011.

\bibitem{Mihai12:Strichartz:Kerr}
M.~Tohaneanu.
\newblock Strichartz estimates on {K}err black hole backgrounds.
\newblock {\em Trans. Amer. Math. Soc.}, 364(2):689--702, 2012.

\bibitem{Mihai22:NW:Kerr}
M.~Tohaneanu.
\newblock Pointwise decay for semilinear wave equations on {K}err spacetimes.
\newblock {\em Lett. Math. Phys.}, 112(1):Paper No. 6, 30, 2022.

\bibitem{yang:scattering:NLW}
S.~Yang.
\newblock Global behaviors of defocusing semilinear wave equations.
\newblock {\em Ann. Sci. \'{E}c. Norm. Sup\'{e}r. (4)}, 55(2):405--428, 2022.

\bibitem{yang:NLW:ptdecay:3D}
S.~Yang.
\newblock Pointwise decay for semilinear wave equations in {$\Bbb{R}^{1 + 3}$}.
\newblock {\em J. Funct. Anal.}, 283(2):Paper No. 109486, 59, 2022.

\end{thebibliography}

  {\footnotesize%
  \addvspace{\medskipamount}
  \textsc{Beijing International Center for Mathematical Research, Peking University, Beijing, China} \par
  \textit{E-mail address}: \texttt{meihe@pku.edu.cn}

  \addvspace{\medskipamount}
  \textsc{School of Mathematical Sciences, Peking University, Beijing, China} \par
  \textit{E-mail address}: \texttt{jnwdyi@pku.edu.cn} \par

  \addvspace{\medskipamount}
  \textsc{Beijing International Center for Mathematical Research, Peking University, Beijing, China} \par
  \textit{E-mail address}: \texttt{shiwuyang@math.pku.edu.cn}
  }

\end{document}